\documentclass[a4paper]{siamart0516}

\usepackage[utf8]{inputenc}
\usepackage[english]{babel}
\usepackage{microtype}
\usepackage{enumitem}
\usepackage{amssymb}
\usepackage{tikz}

\newsiamthm{algorithm2}{Algorithm}
\newsiamremark{example}{Example}
\newsiamremark{assumption}{Assumption}

\newcommand{\st}{\textnormal{s.t.}}
\DeclareMathOperator{\dist}{dist}

\title{
	Augmented Lagrangian Methods for the Solution of Generalized 
	Nash Equilibrium Problems
	\thanks{This research was supported by the German Research Foundation (DFG) under
		grant number KA 1296/24-1 within the priority program "Non-smooth and Complementarity-based Distributed Parameter Systems: Simulation and Hierarchical Optimization" (SPP 1962).}
}

\date{July 14, 2016}

\author{
	Christian Kanzow
	\thanks{University of W\"urzburg,
		Institute of Mathematics, Campus Hubland Nord, Emil-Fischer-Str.\ 30, 
		97074 W\"urzburg, Germany (\email{kanzow@mathematik.uni-wuerzburg.de}).}
	\and
	Daniel Steck
	\thanks{University of W\"urzburg,
		Institute of Mathematics, Campus Hubland Nord, Emil-Fischer-Str.\ 30, 
		97074 W\"urzburg, Germany (\email{daniel.steck@mathematik.uni-wuerzburg.de}).}
}

\begin{document}

\maketitle

\begin{abstract}
\noindent
We propose an augmented Lagrangian-type algorithm for the solution of 
generalized Nash equilibrium problems (GNEPs). Specifically, we discuss 
the convergence properties with regard to both feasibility and optimality 
of limit points. This is done by introducing a secondary GNEP as a new
optimality concept. In this context, special consideration is given to the 
role of suitable constraint qualifications that take into account the
particular structure of GNEPs. Furthermore, we consider the behaviour of 
the method for jointly-convex GNEPs and describe a modification which is 
tailored towards the computation of variational equilibria. Numerical
results are included to illustrate the practical performance of the overall
method.
\end{abstract}



\section{Introduction}

We consider the generalized Nash equilibrium problem which consists of 
$ N $ players, where each player $ \nu = 1, \ldots, N $ tries to solve
his optimization problem 
\begin{equation}\label{Eq:SimpleGNEP}
   \min_{x^{\nu}}\ \theta_{\nu}(x) \quad\st\quad c^{\nu} (x) \le 0,
\end{equation}
where $ \theta_{\nu} : \mathbb R^n \to \mathbb R $ denotes the objective
or utility function of player $ \nu $, $ c^{\nu}: \mathbb R^n \to
\mathbb R^{r_{\nu}} $ defines the constraints, and the vector 
$ x $ consists of the block components $ x^{\nu} \in \mathbb R^{n_{\nu}}, 
\nu = 1, \ldots, N $. These block vectors $ x^{\nu} $ denote the 
variables of player $ \nu $, and we subsume the remaining blocks
into the subvector $ x^{- \nu} $, and then sometimes write 
$ x = (x^{\nu}, x^{-\nu}) $ to indicate the importance of the block vector
$ x^{\nu} $ within the whole vector $ x $. Note that we have
$ n = n_1 + \ldots + n_N $; furthermore, we set $ r := r_1 + \ldots + r_N $
for the total number of constraints. The GNEP is called {\em player-convex}
if all functions $ \theta_{\nu} (\cdot, x^{-\nu}) $ and $ c_i^{\nu} (\cdot,
x^{-\nu} ) $ are convex for any given $ x^{-\nu} $, whereas the GNEP
is called {\em jointly-convex} if, again, the utility functions 
$ \theta_{\nu} $ are convex as a mapping of $ x^{\nu} $ and the
constraints coincide for all players, i.e.\ $ c^1 = \ldots = c^N =: c $, 
and $ c $ is convex as a function of the entire vector $ x $. Note that
the GNEP reduces to the standard {\em Nash equilibrium problem} (NEP) in 
the special case where $ c^{\nu} $ depends on the subvector $ x^{\nu} $
only.

Using this notation, we recall that $ \bar x = \big( \bar x^1, \ldots, 
\bar x^N \big) $ is a {\em (generalized) Nash equilibrium} or simply
a {\em solution} of the GNEP if $ \bar x $ satisfies all the constraints
and, in addition, for each player $ \nu = 1, \ldots, N $, it holds that
\begin{equation*}
   \theta_{\nu} (\bar x) \leq \theta_{\nu} (x^{\nu}, \bar x^{-\nu})
   \quad \forall x^{\nu}: c^{\nu} (x^{\nu}, \bar x^{-\nu}) \leq 0,
\end{equation*}
i.e., $ \bar x $ is a solution if and only if no player $ \nu $ can
improve his situation by unilaterally changing his strategy.

Note that we do not include equality constraints in our GNEP simply
for the sake of notational convenience; our subsequent approach can
easily be extended to equality and inequality constraints. Apart from
this, the above setting is very general since, so far, we do not assume
any convexity assumptions on the mappings $ \theta_{\nu} $ and $ c^{\nu} $
as is done in many other GNEP papers where only the player-convex
or jointly-convex case is considered, cf.\ \cite{Bensoussan1974,Dreves2012,
	Dreves2011,Facchinei2007,Facchinei2010,Fischer2014,Heusinger2009}
for more details. It follows that our framework can, in principle, be
applied to very general classes of GNEPs.

In the meantime, there exist a variety of methods for the solution
of GNEPs, though most of them are designed for player- or
jointly-convex GNEPs and therefore do not cover the GNEP in its
full generality. We refer the interested reader once again to the two
survey papers \cite{Facchinei2010,Fischer2014} and the references
therein for a quite complete overview of the existing approaches.
One of the main problems when solving a GNEP is an inherent singularity
property that arises when some players share the same constraints,
see \cite{Facchinei2009} for more details. Hence, second-order methods
with fast local convergence are difficult to design. This also motivates
us to consider methods which may not be locally superlinearly or 
quadratically convergent, but have nice global convergence properties.

Penalty-type schemes belong to this class of methods. The first
penalty method for GNEPs that we are aware of is due to Fukushima
\cite{Fukushima2011}. A related penalty algorithm was also proposed in
\cite{Facchinei2010a}, and a modification of this algorithm is described in 
\cite{Facchinei2011} where only some of the constraints are penalized. While 
all these approaches prove exactness results under suitable assumptions,
they suffer from the drawback that the resulting penalized subproblems are
nonsmooth Nash equilibrium problems and therefore difficult to solve numerically.

Taking this into account, it is natural to apply an augmented Lagrangian-type
approach in order to solve GNEPs because the resulting subproblems then
have a higher degree of smoothness and should therefore be easier to solve.
This idea is not completely new since Pang and Fukushima \cite{Pang2005} 
applied this idea to quasi-variational inequalities (QVIs). An
improved version of that method can be found in \cite{Kanzow2016}, also
for QVIs. Since the GNEP is a special instance of a QVI, these two
papers also discuss the GNEP within their general QVI-framework. Here,
we apply the augmented Lagrangian idea directly to GNEPs. It turns out that
the corresponding results are significantly stronger than or simply different
from those that arise from the QVI-framework in \cite{Kanzow2016,Pang2005}.
In particular, no GNEP-tailored constraint qualifications are 
considered in \cite{Kanzow2016,Pang2005}, and the feasibility issue,
which plays a central role in this paper, is not discussed there as
a separate topic.

Recall that the augmented Lagrangian (or multiplier-penalty) method is one
of the traditional methods for the solution of constrained optimization 
problems \cite{Bertsekas1982,Nocedal2006} which have also been the subject
of some recent research with several improved convergence results, see, e.g.,
\cite{Birgin2014} and references therein. We therefore try to adapt these
recent improvements to GNEPs in order to get a better understanding
of the augmented Lagrangian approach applied to GNEPs. It turns out,
however, that some results are different from those that are known
for standard optimization problems.

This paper is organized as follows. In \Cref{Sec:CQs}, we deal with 
GNEP-tailored constraint qualifications (CQs), prove some basic results and 
present an error bound as an application. \Cref{Sec:Algorithm} contains 
a precise statement of our algorithm; starting with that section, we divide 
the constraint functions $ c^{\nu} $ into two parts and penalize only one of 
these two parts within our (partial) augmented Lagrangian approach. Hence, 
we consider a whole class of methods which is quite flexible and can take
into account the special structure of the underlying GNEP in a very
favourable way. \Cref{Sec:Convergence} is then dedicated to a 
thorough convergence analysis. To this end, we consider both the feasibility 
and optimality of limit points of our algorithm; in particular, we introduce 
a secondary GNEP called {\em Feasibility GNEP} as a new optimality concept
for generalized Nash games which may be viewed as an interesting counterpart 
of a feasibility result for limit points in the optimization framework, see
\cite{Birgin2014}. In \Cref{Sec:Variational}, we describe how to modify 
our algorithm in a way that is tailored to the computation of variational 
equilibria for jointly-convex GNEPs, and state corresponding convergence 
theorems. \Cref{Sec:Numerics} presents some numerical results, and we 
conclude with some final remarks in \Cref{Sec:Final}.

Notation: Given a function $f=f(x)$ of suitable dimension, we denote by 
$\nabla f$ the transposed Jacobian of $f$. If $x^{\nu}$ is a given subvector 
of $x$, then $\nabla_{x^{\nu}}f$ denotes the submatrix of $\nabla f$ which 
corresponds to the components $x^{\nu}$. Furthermore, given a scalar
$ \alpha $, we write $ \alpha_+ $ for $ \max \{ 0, \alpha \} $. Similarly,
given a vector $v$, we write $ v_+ $ for the vector where the plus-operator
is applied component-wise. When dealing with a function, we occasionally
also write $f_+(x)=(f(x))_+$. All vector norms without an index are
Euclidean norms; the induced matrix norm is denoted by the same symbol.

\section{GNEP Constraint Qualifications}\label{Sec:CQs}

This section is dedicated to an analysis of constraint qualifications for 
GNEPs and their properties. Before we do so, we first recall the definition
of a KKT point.

\begin{definition}\label{Dfn:KKT}
	A pair $(x,\lambda)\in\mathbb{R}^{n+r}$ is called a KKT point
	of the GNEP \eqref{Eq:SimpleGNEP} if
	\begin{equation*}
		\nabla_{x^{\nu}}\theta_{\nu}(x) + \nabla_{x^{\nu}}c^{\nu}(x)\lambda^{\nu}=0
		\quad \text{and} \quad \min\{-c^{\nu}(x),\lambda^{\nu}\}=0
	\end{equation*}
	for every $\nu$. We call $x$ a KKT point if $(x,\lambda)$ is a KKT point for some $\lambda\in\mathbb{R}^m$.
\end{definition}

\noindent
Note that $\min\{-c^{\nu}(x),\lambda^{\nu}\} =0 $ is equivalent to 
$c^{\nu}(x)\le 0$, $\lambda^{\nu}\ge 0$ and $c^{\nu}(x)^T\lambda^{\nu}=0$.

In the theory of augmented Lagrangian methods for optimization problems, two
constraint qualifications have proven to be particularly important: the (extended)
Mangasarian-Fromovitz constraint qualification and the constant positive linear
dependence condition (see \cite{Birgin2014,Qi2000}). Here, we present suitable 
extensions of these conditions to the GNEP setting.

\subsection{Constraint Qualifications}

Recall that we have a GNEP of the form \eqref{Eq:SimpleGNEP}. The first 
condition we present is a GNEP-tailored version of CPLD. Note that we 
call a collection of vectors $v_1,\ldots,v_k$ \emph{positively linearly 
dependent} if the system $ \lambda_1 v_1 + \ldots + \lambda_k v_k=0, \
\lambda_i\ge 0 $, has a nontrivial solution. Otherwise, the vectors are called 
\emph{positively linearly independent}.

\begin{definition}\label{Dfn:CPLD}
Consider a GNEP of the form \eqref{Eq:SimpleGNEP}. Let $\nu$ be a given 
index and $x\in\mathbb{R}^n$ be a given point with $c^{\nu}(x)\le 0$. We 
say that $c^{\nu}$ satisfies {\em CPLD with respect to player $\nu$} or simply 
{\em CPLD$_{\nu}$} if, whenever the partial gradients 
$ \nabla_{x^{\nu}}c_i^{\nu}(x) \ ( i\in I ) $
are positively linearly dependent for some subset
$ I\subset\{i\in\{1,\ldots, r_{\nu} \}~|~c_i^{\nu}(x)=0\} $,
the same gradients are linearly dependent in a neighbourhood of $x$. 
Moreover, we say that the GNEP \eqref{Eq:SimpleGNEP} satisfies {\em GNEP-CPLD}
in $x$ if, for every $\nu\in\{1,\ldots,N\}$, the function $c^{\nu}$ 
satisfies CPLD$_{\nu}$ in $x$.
\end{definition}

\noindent
In the simplest case $N=1$ (i.e.\ there is only one player), the above reduces
to the classical CPLD, cf.\ \cite{Qi2000}. Hence, one might consider GNEP-CPLD as
a straightforward generalization of CPLD to the multi-player setting. However, there
are some peculiarities that need to be pointed out. Clearly, the above condition
only makes an assertion about the partial gradients with regard to the respective
player's variable $x^{\nu}$. However, we require that the positive linear
dependence (if there is one) extends to a whole neighbourhood of $x$. This
makes \cref{Dfn:CPLD} a condition which should not be attributed
to each player $\nu$ but rather to the GNEP as a whole.

We now define an analogue of the extended MFCQ. Here, we do not require 
the point $x$ to be feasible, hence the term \emph{extended} MFCQ.

\begin{definition}\label{Dfn:EMFCQ}
Consider a GNEP of the form \eqref{Eq:SimpleGNEP}. Let $\nu$ be a given 
index and $x\in\mathbb{R}^n$ be a given point. We say that $c^{\nu}$ 
satisfies {\em EMFCQ with respect to player $\nu$} or simply 
{\em EMFCQ$_{\nu}$} if there is a vector $d^{\nu}\in\mathbb{R}^{n_{\nu}}$ 
such that
\begin{equation*}
   c_i^{\nu}(x)\ge 0 \implies \nabla_{x^{\nu}}c_i^{\nu}(x)^T d^{\nu}<0
\end{equation*}
holds for every $i\in\{1,\ldots,r_{\nu}\}$. Moreover, we say that the 
GNEP \eqref{Eq:SimpleGNEP} satisfies {\em GNEP-EMFCQ} in $x$ if, for every 
$\nu\in\{1,\ldots,N\}$, the function $c^{\nu}$ satisfies EMFCQ$_{\nu}$ in $x$.
\end{definition}

\noindent
While GNEP-CPLD seems to be a new constraint qualification for GNEPs,
the GNEP-EMFCQ condition is already used in \cite{Facchinei2010a,Fukushima2011} 
to prove exactness results for suitable penalty methods; apart from this,
these references do not contain any further discussion of GNEP-EMFCQ.
Since both constraint qualifications play a central role in our
subsequent analysis, we therefore discuss their main properties in
this section.

To this end, first note that EMFCQ boils down to the classical MFCQ condition
in case of feasible points $ x $. Hence, when dealing with feasible points,
we will sometimes simply write GNEP-MFCQ instead of GNEP-EMFCQ. By use of a 
classical theorem of the alternative, it is easy to see that \cref{Dfn:EMFCQ}
can equivalently be stated as the gradients $\nabla_{x^{\nu}}c_i^{\nu}(x)\ 
(c_i^{\nu}(x)\ge 0) $ being positively linearly independent. This immediately
shows that GNEP-MFCQ (for feasible points) implies GNEP-CPLD.

Clearly, the above two CQs are conditions which are tailored to GNEPs. 
However, it is not immediately clear whether there is a relationship between 
the "classical" constraint qualifications and their GNEP counterparts. 
In fact, one could simply concatenate the player constraints $c^{\nu}$ 
into one mapping
\begin{equation}\label{Eq:ConcatenatedConstraints}
   c(x)=
   \begin{pmatrix}
      c^1(x) \\
      \vdots \\
      c^N(x)
   \end{pmatrix}
\end{equation}
and ask whether we can reduce GNEP constraint qualifications to conditions 
for this function. In general, however, this is not possible. To this end, 
consider the following set of examples.

\begin{example}
In both examples, we have two players $ \nu = 1, 2 $ with $ n_1 = n_2 := 1 $,
and the mapping $ c $ is defined by \eqref{Eq:ConcatenatedConstraints}
with $ r_1 = r_2 := 1 $. To simplify the notation, we write $ c_1 $ and 
$ c_2 $ instead of $ c_1^1 $ and $ c_1^2 $, respectively, for the two
components of $ c $.
\begin{enumerate}[label=\textnormal{(\alph*)}]
   \item Consider the function
      \begin{equation*}
         c(x_1,x_2)=
	 \begin{pmatrix}
	    x_1 \\
	    x_1 + x_2^2
	 \end{pmatrix}
      \end{equation*}
      and the point $\bar{x}=(0,0)$. Using $d=(-1,0)$, it follows 
      that $\nabla c_1(\bar{x})^T d<0$ and 
      $\nabla c_2(\bar{x})^T d<0$. Hence, standard EMFCQ holds for 
      this constraint. However, we have $\nabla_{x_2}c_2(\bar{x})=0$, 
      which means that EMFCQ$_2$ cannot hold. In fact, even CPLD$_2$ is 
      not satisfied since $\nabla_{x_2}c_2(x)=2 x_2$ for all $x\in\mathbb{R}^2$.
   \item Consider the function
      \begin{equation*}
         c(x_1,x_2)=
         \begin{pmatrix}
            2 x_1 - x_2^2 - 1 \\
            2 x_2 - x_1^2 - 1
         \end{pmatrix}
      \end{equation*}
      and the point $\bar{x}=(1,1)$. Due to $\nabla_{x_1}c_1(x)=
      \nabla_{x_2}c_2(x)=2$, it is clear that GNEP-EMFCQ holds in 
      $\bar{x}$. On the other hand, the gradients of $c$ are given by
      \begin{equation*}
         \nabla c_1(x)=
         \begin{pmatrix}
            2 \\
            -2 x_2
         \end{pmatrix},\quad
         \nabla c_2(x)=
         \begin{pmatrix}
            -2 x_1 \\
            2
         \end{pmatrix}.
      \end{equation*}
      This shows that $c$ satisfies neither EMFCQ nor CPLD in $\bar{x}$.
\end{enumerate}
\end{example}

\noindent
These examples show that, in general, the classical CPLD and EMFCQ are 
entirely different conditions in comparison to their GNEP counterparts. 
There is, however, an important special case which arises if the functions 
$c^{\nu}$ depend on $x^{\nu}$ only, so we have a standard NEP. In this case, 
the transposed Jacobian $\nabla c(x)$ is a block diagonal matrix of the form
\begin{equation}\label{Eq:ConcatenatedConstraintsGrad}
   \nabla c(x)=
   \begin{pmatrix}
      \nabla_{x^1} c^1(x^1) & & \\
      & \ddots & \\
      & & \nabla_{x^{N}}c^N(x^N) \\
   \end{pmatrix}
   \quad \text{with} \quad
   \nabla_{x^{\nu}} c^{\nu} (x^{\nu}) \in \mathbb R^{n_{\nu} \times r_{\nu}} .
\end{equation}
This makes it easy to prove that GNEP-CPLD is equivalent to CPLD 
(for the function $c$), and the same holds with CPLD replaced by EMFCQ.

\begin{theorem}
Consider a standard NEP of the form \eqref{Eq:SimpleGNEP} with $C^1$-functions $\theta_{\nu}$ and
$c^{\nu}$. If $\bar{x}\in\mathbb{R}^n$ is a given point, then the following assertions are true:
\begin{enumerate}[label=\textnormal{(\alph*)}]
   \item If $\bar{x}$ is feasible, then GNEP-CPLD holds in $\bar{x}$
      iff $c$ satisfies CPLD in $\bar{x}$.
   \item GNEP-EMFCQ holds in $\bar{x}$ iff $c$ satisfies 
      EMFCQ in $\bar{x}$.
\end{enumerate}
\end{theorem}

\begin{proof}
	The proof is based on \eqref{Eq:ConcatenatedConstraintsGrad} and is rather straightforward.
\end{proof}

\noindent
We now prove two theorems which establish the role of GNEP-CPLD and 
GNEP-EMFCQ as constraint qualifications. These theorems play a fundamental 
role in our analysis and will be referenced multiple times later on. 
It should be noted, however, that the proofs are obtained by suitable 
adaptations of the corresponding proofs for classical optimization problems.

\begin{theorem}\label{Thm:SequentialCPLDnu}
Consider a GNEP of the form \eqref{Eq:SimpleGNEP} where $\theta_{\nu}$ 
and $c^{\nu}$ are $C^1$-functions. Let $(x^k)\subset\mathbb{R}^n$ be a 
sequence converging to $\bar{x}$ and $(\lambda^{\nu,k})\subset
\mathbb{R}^{r_{\nu}}$ be vectors with
\begin{equation}\label{Eq:ApproxKKT}
   \nabla_{x^{\nu}}\theta_{\nu}(x^k) + \nabla_{x^{\nu}}c^{\nu}(x^k)\lambda^{\nu,k}
   \to 0 \quad \text{and} \quad
   \min\{-c^{\nu}(x^k),\lambda^{\nu,k}\}\to 0
\end{equation}
for every $\nu$. If GNEP-CPLD holds in $\bar{x}$, then 
$\bar{x}$ together with some multiplier $\bar{\lambda}$ 
is a KKT point of the GNEP.
\end{theorem}

\begin{proof}
Let $\nu\in\{1,\ldots,N\}$. Since the relations \eqref{Eq:ApproxKKT}
remain true if we replace $ \lambda^{\nu,k} $ by $ \lambda^{\nu,k}_+ $,
we may assume, without loss of generality, that $\lambda^{\nu,k}\ge 0$ for 
all $k$. Furthermore, we have $\lambda_i^{\nu,k}\to 0$ for every $i$ with 
$c_i^{\nu}(\bar{x})<0$. Hence, we get
\begin{equation*}
   \nabla_{x^{\nu}}\theta_{\nu}(x^k) + \sum_{c_i^{\nu}(\bar{x})=0}
   \lambda_i^{\nu,k}\nabla_{x^{\nu}}c_i^{\nu}(x^k)\to 0.
\end{equation*}
Using a Carath\'eodory-type result, cf.\ \cite[Lem.\ 3.1]{Birgin2014}, we 
can choose subsets
\begin{equation*}
   I^{\nu,k}\subset\{i~|~c_i^{\nu}(\bar{x})=0\}
\end{equation*}
such that the gradients $\nabla_{x^{\nu}}c_i(x^k) \ (i\in I^{\nu,k})$ are 
linearly independent and we can write
\begin{equation*}
   \sum_{c_i^{\nu}(\bar{x})=0}\lambda_i^{\nu,k}\nabla_{x^{\nu}}c_i^{\nu}(x^k)=
   \sum_{i\in I^{\nu,k}}\hat{\lambda}_i^{\nu,k}\nabla_{x^{\nu}}c_i^{\nu}(x^k)
\end{equation*}
for some vectors $\hat{\lambda}^{\nu,k}\ge 0$. Subsequencing if necessary, 
we may assume that $I^{\nu,k}=I^{\nu}$ for every $k$, i.e.\ we get
\begin{equation}\label{Eq:Using0}
   \nabla_{x^{\nu}}\theta_{\nu}(x^k) + \sum_{i\in I^{\nu}}
   \hat{\lambda}_i^{\nu,k}\nabla_{x^{\nu}}c_i^{\nu}(x^k)\to 0.
\end{equation}
We claim that the sequence $(\hat{\lambda}^{\nu,k})$ is bounded. 
If this is not the case, then we can divide both sides of the 
above equation by $\|\hat{\lambda}^{\nu,k}\|$, take the limit $k\to\infty$ 
on a suitable subsequence and obtain a nontrivial positive linear 
combination of the gradients $\nabla_{x^{\nu}}c_i(\bar{x})$, $i\in I^{\nu}$, 
which vanishes. Hence, by CPLD, these gradients should be linearly dependent 
in a neighbourhood of $\bar{x}$, which is a contradiction.

Hence $(\hat{\lambda}^{\nu,k})$ is bounded; let 
$ \bar{\lambda}_i^{\nu} \ (i \in I^{\nu}) $ be a limit point.
Setting $ \bar{\lambda}_i^{\nu} := 0 $ for all $ i \not\in I^{\nu} $, and
taking into account \eqref{Eq:Using0}, it follows that $ \bar{x} $
together with the multiplier $ \bar{\lambda}^{\nu} $ satisfies the
KKT conditions of player $ \nu $. Since $ \nu \in \{ 1, \ldots, N \} $
was chosen arbitrarily, the statement follows.
\end{proof}

\noindent
Note that assumption \eqref{Eq:ApproxKKT} means that $ x^k $, together
with some multiplier estimate $ \lambda^{\nu,k} $, satisfies the 
KKT conditions of player $ \nu $ inexactly. In contrast to the
approximate KKT conditions used in \cite{Birgin2014} (also applied
in \cite{Kanzow2016}), however, we do not assume that the multiplier
estimates are nonnegative which gives some more freedom in our choice
of methods for computing approximate KKT points. Furthermore, let us
mention explicitly that \eqref{Eq:ApproxKKT} automatically implies that
any limit point of the sequence $ (x^k) $ is at least feasible for the
GNEP \eqref{Eq:SimpleGNEP}.

We also stress that, as is usually the case with CPLD-type conditions, 
the $\nu$-th component of the vector $\bar{\lambda}$ is not 
necessarily a limit point of the sequence $(\lambda^{\nu,k})$. 
This property is, in general, only true if we assume a stronger 
constraint qualification. To this end, consider the following theorem 
which uses GNEP-MFCQ (recall the feasibility of the limit points, hence
there is no need to assume GNEP-EMFCQ).

\begin{theorem}\label{Thm:SequentialMFCQnu}
Consider a GNEP of the form \eqref{Eq:SimpleGNEP} where $\theta_{\nu}$ and
$c^{\nu}$ are $C^1$-functions. Let $(x^k)\subset\mathbb{R}^n$ be a sequence
converging  to $\bar{x}$ and $(\lambda^{\nu,k})\subset\mathbb{R}^{r_{\nu}}$
be vectors such that \eqref{Eq:ApproxKKT} holds for every $\nu$. If GNEP-MFCQ
holds in $\bar{x}$, then the sequences $(\lambda^{\nu,k})$ are bounded.
Moreover, if $\bar{\lambda}^{\nu}$ is a limit point of $(\lambda^{\nu,k})$
for every $\nu$, then $\bar{x}$ together with $\bar{\lambda}=
(\bar{\lambda}^1,\ldots,\bar{\lambda}^N)$ is a KKT point of the GNEP.
\end{theorem}

\begin{proof}
Clearly, it suffices to show the boundedness. To this end, let 
$\nu\in\{1,\ldots,N\}$ be an arbitrary player. By assumption, we have
$\lambda_i^{\nu,k}\to 0$ for every $i$ with $c_i^{\nu}(\bar{x})<0$. 
Hence, recalling that $\bar{x}$ is feasible by \eqref{Eq:ApproxKKT},
we get
\begin{equation*}
   \nabla_{x^{\nu}}\theta_{\nu}(x^k)+\sum_{c_i^{\nu}(\bar{x})=0}
   \lambda_i^{\nu,k}\nabla_{x^{\nu}}c_i^{\nu}(x^k)\to 0.
\end{equation*}
Assume now, by contradiction, that $\|\lambda^{\nu,k}\|\to\infty$. 
Dividing the above equation by $\|\lambda^{\nu,k}\|$, we obtain
\begin{equation*}
   \sum_{c_i^{\nu}(\bar{x})=0}\alpha_i^{\nu,k}\nabla_{x^{\nu}}c_i^{\nu}(x^k)\to 0, 
   \quad\text{where}\quad\alpha^{\nu,k}=
   \frac{\lambda^{\nu,k}}{\|\lambda^{\nu,k}\|}.
\end{equation*}
Obviously, $(\alpha^{\nu,k})$ is bounded and has a limit point $\alpha^{\nu}$ 
with $\alpha^{\nu}\ge 0$ and $\|\alpha^{\nu}\|=1$. Hence, we obtain
\begin{equation*}
   \sum_{c_i^{\nu}(\bar{x})=0}\alpha_i^{\nu}
   \nabla_{x^{\nu}}c_i^{\nu}(\bar{x})=0,
\end{equation*}
which contradicts GNEP-MFCQ.
\end{proof}

\noindent
The previous results indicate that GNEP-MFCQ is a more practical property 
than GNEP-CPLD, because it allows us to explicitly construct the 
multipliers which make $\bar{x}$ a KKT point. However, when dealing 
with approximate KKT conditions of the type
\begin{equation}\label{Eq:ApproxKKT2}
   \nabla_{x^{\nu}}\theta_{\nu}(x^k) + \nabla_{x^{\nu}}c^{\nu}(x^k)\lambda^{\nu,k}
   \to 0
\end{equation}
we will typically use an inexact stopping criterion. That is, we stop the 
iteration as soon as the left-hand side of the above equation is 
sufficiently close to zero, regardless of whether $\lambda^{\nu,k}$ is 
close to a multiplier $\bar{\lambda}^{\nu}$ which satisfies
$ \nabla_{x^{\nu}}\theta_{\nu}(\bar{x}) + 
\nabla_{x^{\nu}}c^{\nu}(\bar{x})\bar{\lambda}^{\nu}=0 $.
It is a peculiarity of GNEP-CPLD that the sequence of multipliers can be 
unbounded, but we still have the approximate KKT 
condition \eqref{Eq:ApproxKKT2}.

\subsection{An Error Bound Result}

There exist different types of error bounds in the optimization literature.
One class of error bounds provides a computable estimate for the distance
of a given point to the solution set or the set of KKT points, the other
class provides a measure for the distance to the feasible set. For GNEPs,
there exist some error bound results of the former type, see the papers
\cite{Dreves2014,Izmailov2014}, whereas here we use our GNEP constraint
qualifications to show that they can be used to obtain an error bound
of the latter type.

To this end, consider a GNEP of the form \eqref{Eq:SimpleGNEP} where 
$c^{\nu}$ is the constraint function of player $\nu$. It will be convenient 
to define the sets
\begin{equation*}
   X_{\nu}(x^{-\nu})=\{x^{\nu}\in\mathbb{R}^{n_{\nu}}~|~c^{\nu}(x^{\nu},x^{-\nu})
   \le 0\}.
\end{equation*}
It is well known that, for classical optimization problems, the CPLD 
constraint qualification implies a local error bound on the feasible set, 
see \cite{Andreani2012}. This result can readily be applied to GNEPs if 
we consider the concatenated constraint function $c$ from 
\eqref{Eq:ConcatenatedConstraints}. This yields an error bound on the 
distance to the set
\begin{equation*}
   X=\{x\in\mathbb{R}^n~|~c(x)\le 0\},
\end{equation*}
i.e.\ the set of points which are feasible for every player. However, 
this set is not natural to GNEPs since it does not preserve the structure 
of the players' individual optimization problems. Furthermore, we cannot 
expect such an error bound to hold without additional requirements on the 
partial gradients $ \nabla_{x^{\mu}}c^{\nu}(x), \ \mu\ne\nu, $
of player $\nu$'s constraint function with respect to another player $\mu$. 
Hence, it is more natural to ask for player-specific error bounds of the form
\begin{equation}\label{Eq:ErrorBoundnu}
   \dist(x^{\nu},X_{\nu}(x^{-\nu}))\le C\|c_+^{\nu}(x)\|,
\end{equation}
which measure the distance of $x^{\nu}$ to the corresponding set 
$X_{\nu}(x^{-\nu})$. Special care needs to be taken because the set 
$X_{\nu}(x^{-\nu})$ could be empty. In fact, this latter point is where the 
theory of GNEP error bounds is substantially different from the 
corresponding theory for classical optimization problems. To see this, 
consider a point $x$ and a player $\nu$ such that $x^{\nu}$ is on the
 boundary of $X_{\nu}(x^{-\nu})$. Two questions need to be considered:
\begin{itemize}
   \item Is there a neighbourhood $U$ of $x$ such that the set 
      $X_{\nu}(y^{-\nu})$ is nonempty for every $y\in U$?
   \item If $y^k$ is a sequence of points converging to $x$, does the 
      sequence of distances $ d^k=\dist(y^{\nu,k},X_{\nu}(y^{-\nu,k})) $
      converge to zero?
\end{itemize}
It is particularly the second question which poses significant difficulties 
to our analysis. In fact, a consequence of these problems is that 
GNEP-CPLD is not strong enough to imply a partial error bound.

\begin{example}
\begin{enumerate}[label=\textnormal{(\alph*)}]
   \item Consider a jointly-convex GNEP with two players, each 
      controlling a single variable. For simplicity, we denote the 
      variables by $x$ and $y$. The constraint function is given by
      $	c^1(x,y)=c^2(x,y)=x $. Clearly, GNEP-CPLD holds at every 
      feasible point, because the constraints are linear. However, given 
      any point $(0,\bar{y})$ on the boundary of the feasible region 
      and a neighbourhood $U$, there are points $(x,y)\in U$ such that 
      $X_2(x)$ is empty. For instance, we can simply choose $(x,y)=
      (\varepsilon,\bar{y})$ for any $\varepsilon>0$, cf.\ 
      \cref{Fig:ErrorBound} (left).
   \item Consider another GNEP with two players, each controlling a 
      single variable. Like above, we write $x$ and $y$. Let player $1$'s 
      (smooth!) constraint function be given by $ c^1(x,y) = y-\min\{0,x\}^2 $.
      Consider the feasible point $(\bar{x},\bar{y})=(1,0)$. 
      The function $c^1$ is linear in a neighbourhood of 
      $(\bar{x},\bar{y})$, which implies that GNEP-CPLD holds. 
      Furthermore, unlike with example (a), the set $X_1(y)$ is nonempty 
      for every $(x,y)$ in a neighbourhood of $(\bar{x},\bar{y})$. 
      Despite this, an error bound does not hold because, given any point 
      $(x,y)=(1,\varepsilon)$ with $\varepsilon>0$, it holds that
      $\dist(x,X_1(y))=1+\sqrt{\varepsilon}$, cf.\ \cref{Fig:ErrorBound} (right).
\end{enumerate}
\end{example}

\begin{figure}\centering
\begin{tikzpicture}
	\draw[very thin,color=gray] (-2.5,0) -- (2.5,0) node [right,color=black] {$x$};
	\draw[very thin,color=gray] (0,-1) -- (0,2.5) node [above,color=black] {$y$};
	
	\fill[color=blue,opacity=0.1]
		(0,2.5) -- (0,-1) -- (-2.5,-1) -- (-2.5,2.5);
	
	\draw[very thick] (0.2,1) circle (1pt) node [right] {$(\varepsilon,\bar{y})$};
	\draw[thin,red,dashed] (0.2,-1) -- (0.2,2.5) node [right] {$X_2(\varepsilon)=\emptyset$};
		
	\draw[very thick] (0,2.5) -- (0,-1);
\end{tikzpicture}
\quad
\begin{tikzpicture}[every plot/.append style={samples=100}]
	\draw[very thin,color=gray] (-2.5,0) -- (2.5,0) node [right,color=black] {$x$};
	\draw[very thin,color=gray] (0,-1) -- (0,2.5) node [above,color=black] {$y$};
	
	\fill[color=blue,opacity=0.1]
		plot[domain=-1.581:0.0] (\x,{\x*\x}) -- (0,0) -- (2.5,0) -- (2.5,-1) -- (-2.5,-1)
		-- (-2.5,2.5) -- (-1.581,2.5);
		
	\draw[very thick] (1,0.2) circle (1pt) node [above] {$(1,\varepsilon)$};
	\draw[thin,dashed] (1,0.2) -- (-0.447,0.2);
	\draw[very thick,red] (-2.5,0.2) node[above] {$X_1(\varepsilon)$} -- (-0.447,0.2);
	
	\draw[very thick] (0,0) -- (2.5,0);
	\draw[very thick,domain=-1.581:0.0] plot(\x,{\x*\x});
\end{tikzpicture}
\caption{Illustration of Example~2.7(a) (left) and (b) (right).}
\label{Fig:ErrorBound}
\end{figure}
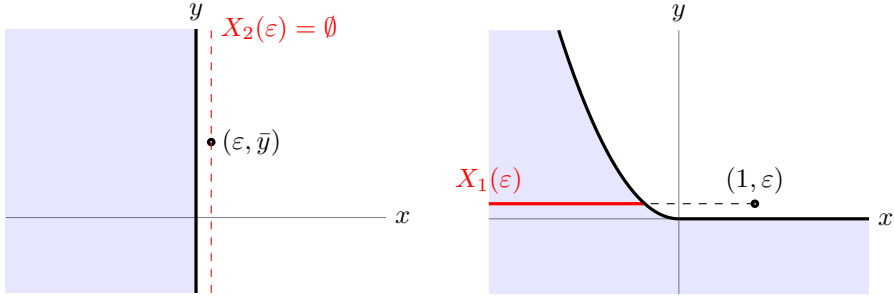

\noindent
Despite this negative result, it turns out that GNEP-MFCQ does imply an 
error bound. In order to show this, we first prove a technical lemma.

\begin{lemma}\label{Lem:ProjectionNu}
For a GNEP of the form \eqref{Eq:SimpleGNEP}, let $\nu$ be a given index, 
let $x$ be a given point with $c^{\nu}(x)\le 0$ and assume that $c^{\nu}$ 
satisfies MFCQ$_{\nu}$ in $x$. Then we have the following properties:
\begin{enumerate}[label=\textnormal{(\alph*)}]
   \item There is a neighbourhood $U$ of $x$ such that, for every $y\in U$, 
      the set $X_{\nu}(y^{-\nu})$ is nonempty.
   \item Given $\varepsilon>0$, we can choose a neighbourhood $U$ of $x$ 
      such that, for every $y\in U$, there is a point 
      $z^{\nu}\in X_{\nu}(y^{-\nu})$ with $\|z^{\nu}-y^{\nu}\|\le\varepsilon$.
\end{enumerate}
\end{lemma}

\begin{proof}
Since statement (b) implies (a), it suffices to show assertion (b).
To this end, let $\varepsilon>0$ be a positive 
number. By MFCQ$_{\nu}$, there is a vector $d^{\nu}\in\mathbb{R}^{n_{\nu}}$ 
such that $ \nabla_{x^{\nu}}c_i^{\nu}(x)^T d^{\nu}<0 $ holds for every $i$ 
with $c_i^{\nu}(x)=0$. By the mean value theorem and the continuity of
$ c^{\nu} $, this implies that, for sufficiently small $t>0$, the point 
$x_t=(x^{\nu}+t d^{\nu},x^{-\nu})$ is strictly feasible for player $ \nu $, 
i.e.\ $ c_i^{\nu} (x^t) < 0 $ for all $ i = 1, \ldots, r_{\nu} $ and all 
$ t > 0 $ sufficiently small. We then choose $t>0$ small enough so that 
$\|x_t-x\|\le\varepsilon/2$ and, subsequently, a radius $r>0$ such that
the (full-dimensional) neighbourhood $B_{r}(x_t)$ consists of feasible
points for player $ \nu $; note that the latter exists by the continuity
of $c^{\nu}$ and the strict feasibility of $ x_t $ for player $\nu$. Now,
set  $r'=\min\{r,\varepsilon/2\}$ and $U=B_{r'}(x)$. We claim that this
set $ U $ has the desired properties. In fact, take an arbitrary element
$ y \in U $, and define $z^{\nu} := x^{\nu}+t d^{\nu}$. Then we have
$(z^{\nu},y^{-\nu})\in B_r(x_t)$ and hence $z^{\nu}\in X_{\nu}(y^{-\nu})$.
Furthermore, $ \| x^{\nu} - y^{\nu}\| \leq \| x - y \| \leq r' \leq
\varepsilon/2 $ and $ \| z^{\nu} - x^{\nu} \|= \| t d^{\nu} \| =
\| x_t - x \| \leq \varepsilon/2 $, hence the triangle inequality
implies $ \| z^{\nu} - y^{\nu} \| \leq \| z^{\nu} - x^{\nu} \| +
\| x^{\nu} - y^{\nu} \| \leq \varepsilon $. This completes the proof.
\end{proof}

\noindent
The above lemma guarantees that, for $y$ in a vicinity of a given point $x$, 
the projection of $y^{\nu}$ onto the feasible set $X_{\nu}(y^{-\nu})$ is 
sufficiently well-behaved. Roughly speaking, if $y$ is close to $x$, then 
there is a feasible point $(z^{\nu},y^{-\nu})$ which is close to $y$ 
(and hence, close to $x$). Note that, in view of the previous examples, 
GNEP-CPLD is not enough to even imply part (a) of the lemma.

\begin{theorem}\label{Thm:ErrorBoundMFCQ}
For a GNEP of the form \eqref{Eq:SimpleGNEP}, let $\nu$ be a given index 
and $x$ be a given point with $c^{\nu}(x)\le 0$. Assume that $c^{\nu}$ 
satisfies MFCQ$_{\nu}$ in $x$ and $\nabla_{x^{\nu}}c^{\nu}$ is 
Lipschitz-continuous in a neighbourhood of $x$. Then there is a 
constant $C>0$ and a neighbourhood $U$ of $x$ such that, for every $y\in U$, 
we have the error bound \eqref{Eq:ErrorBoundnu}.
\end{theorem}

\begin{proof}
By \cref{Lem:ProjectionNu} (a), there is a neighbourhood $\tilde{U}$ 
of $x$ such that, for every $y\in\tilde{U}$, the set $X_{\nu}(y^{-\nu})$ is 
nonempty. By the local Lipschitz continuity of $\nabla_{x^{\nu}}c^{\nu}$, we can 
choose $\tilde{U}$ small enough so that there is a constant $C_1>0$ with
\begin{equation}\label{Eq:TaylorLip}
   c_i^{\nu}(z)+\nabla_{z^{\nu}}c_i^{\nu}(z)^T(y^{\nu}-z^{\nu})\le c_i^{\nu}(y)+
   C_1 \|z^{\nu}-y^{\nu}\|^2
\end{equation}
for every $i=1,\ldots,r_{\nu}$ and $y,z\in \tilde{U}$ with $y^{-\nu}=z^{-\nu}$. 
Now, let $y\in\tilde{U}$ be an infeasible point for player $ \nu $ (for 
feasible points, there is nothing to prove), and let $z^{\nu}=z^{\nu}(y)$ be 
a projection of $y^{\nu}$ onto the (nonempty and closed, but not necessarily
convex) set $X_{\nu}(y^{-\nu})$, i.e.\ $ z^{\nu} $ is a solution of
the optimization problem
\begin{equation}\label{Eq:ProjOpt}
   \min \| \xi^{\nu} - y^{\nu} \| \quad \st \quad 
   c^{\nu} (\xi^{\nu}, y^{-\nu}) \leq 0.
\end{equation}
For brevity, we write $z=(z^{\nu},y^{-\nu})$. Since MFCQ$_{\nu}$ holds at
$ x $, this condition also holds in a neighbourhood of $ x $. Taking
into account \cref{Lem:ProjectionNu} (b), it follows that the point
$ z $ is arbitrarily close to $ y $. Hence, without loss of generality,
we can assume that MFCQ$_{\nu}$ holds at $z$. It then follows that $z^{\nu}$ 
satisfies the KKT condition
\begin{equation*}
   \frac{z^{\nu}-y^{\nu}}{\|z^{\nu}-y^{\nu}\|}+\nabla_{z^{\nu}}c^{\nu}(z)
   \lambda^{\nu}=0
\end{equation*}
of the  optimization problem \eqref{Eq:ProjOpt}, where 
$\lambda^{\nu}=\lambda^{\nu}(y)\in\mathbb{R}^{r_{\nu}}$ denotes a
corresponding (nonnegative) Lagrange multiplier. Premultiplying this 
equation by $ (z^{\nu} - y^{\nu})^T $ yields
\begin{equation*}
   \|z^{\nu}-y^{\nu}\| = (\lambda^{\nu})^T\nabla_{z^{\nu}} c^{\nu}(z)^T 
   (y^{\nu}-z^{\nu}) \leq \sum_{i\in I^{\nu}} \lambda_i^{\nu} 
   \nabla_{z^{\nu}}c_i^{\nu}(z)^T(y^{\nu}-z^{\nu}),
\end{equation*}
where $I^{\nu}=I^{\nu}(y)$ is the set of indices for which the corresponding 
term in the sum is positive. Since $\lambda_i^{\nu}\ge 0$ for all
$i=1,\ldots,r_{\nu}$, this implies $\lambda_i^{\nu}>0 $ and
$\nabla_{z^{\nu}}c_i^{\nu}(z)^T(y^{\nu}-z^{\nu})>0 $ for all $i\in I^{\nu}$.
In particular, we have $c_i^{\nu}(z)=0$ for every $i\in I^{\nu}$. Furthermore, 
\cref{Thm:SequentialMFCQnu} implies the existence of a constant 
$C_2>0$ such that $\|\lambda^{\nu}(y)\|\le C_2$ for every $y\in\tilde{U}$. 
We now apply \cref{Lem:ProjectionNu} (b) with $\varepsilon=(2 r_{\nu} C_1 C_2)^{-1}$
and obtain a neighbourhood $U\subset\tilde{U}$ of $x$ with
$\|z^{\nu}(y)-y^{\nu}\|\le\varepsilon$ for every $y\in U$. It follows that
\begin{eqnarray*}
   \|z^{\nu}-y^{\nu}\| & \le & \sum_{i \in I^{\nu}} 
   \underbrace{\lambda_i^{\nu}}_{\leq C_2} 
   \underbrace{\nabla_{z^{\nu}} c_i^{\nu} (z)^T \big( y^{\nu} - z^{\nu} 
   \big)}_{> 0} \\
   & \leq & C_2 \sum_{i \in I^{\nu}} \big( \underbrace{c_i^{\nu} (z)}_{=0} +
   \nabla_{z^{\nu}} c_i^{\nu} (z)^T \big( y^{\nu} - z^{\nu} \big) \big) \\
   & \stackrel{\eqref{Eq:TaylorLip}}{\leq} & C_2 \sum_{i\in I^{\nu}} 
   \left( c_i^{\nu}(y) + C_1 \|z^{\nu}-y^{\nu}\|^2 \right) \\
   & \leq & C_2 \sum_{i \in I^{\nu}} c_i^{\nu} (y) + C_1 C_2 r_{\nu}
   \| z^{\nu} - y^{\nu} \|^2
\end{eqnarray*}
and hence
\begin{equation*}
   \|z^{\nu}-y^{\nu}\|-r_{\nu} C_1 C_2 \|z^{\nu}-y^{\nu}\|^2\le 
   C_2 \sum_{i\in I^{\nu}} c_i^{\nu}(y)\le C_2 \|c_+^{\nu}(y)\|_1
\end{equation*}
for every $y\in U$. This implies the desired error bound, since
\begin{equation*}
   \|z^{\nu}-y^{\nu}\|-r_{\nu} C_1 C_2 \|z^{\nu}-y^{\nu}\|^2
   \ge\frac{1}{2}\|z^{\nu}-y^{\nu}\|
\end{equation*}
by the definition of $\varepsilon$.
\end{proof}

\noindent
The above theorem establishes player-individual error bounds for GNEPs 
which satisfy GNEP-MFCQ. Note that this does not imply an error bound to 
the set $X$ of points which are feasible for the GNEP as a whole. In fact, 
the latter set could be empty and \cref{Thm:ErrorBoundMFCQ} still holds.

\section{An Augmented Lagrangian Method}\label{Sec:Algorithm}

This section describes an augmented Lagrangian method for GNEPs. Due to the 
nature of our penalization scheme, we have decided to adjust the notation 
in a manner that accounts for the possibility of partial penalization. 
To this end, we replace the constraint functions $c^{\nu}$ from 
\eqref{Eq:SimpleGNEP} by pairs of functions
\begin{equation*}
   c^{\nu} = \begin{pmatrix} g^{\nu} \\ h^{\nu} \end{pmatrix} \quad
   \text{with} \quad
   g^{\nu}: \mathbb R^n \to \mathbb R^{m_{\nu}}, \quad
   h^{\nu}: \mathbb R^n \to \mathbb R^{p_{\nu}} \quad (\text{i.e.\ }
   r_{\nu} = m_{\nu} + p_{\nu})
\end{equation*}
both of which are assumed to be at least continuously differentiable. 
Similarly to the previous notation, we write
\begin{equation*}
   m:=m_1+\ldots+m_N, \qquad p:=p_1+\ldots+p_N
\end{equation*}
and consider a GNEP where player $\nu$ has to solve the optimization problem
\begin{equation}\label{Eq:OriginalGNEP}
   \min_{x^{\nu}} \ \theta_{\nu} (x) \quad
   \st \quad g^{\nu}(x)\le 0,\quad h^{\nu}(x)\le 0.
\end{equation}
In principle, this is exactly the same problem as \eqref{Eq:SimpleGNEP}. 
However, the two functions $g^{\nu}$ and $h^{\nu}$ play completely different 
roles in our method. More precisely, $g^{\nu}$ describes the set of 
constraints which we will penalize, whereas $h^{\nu}$ is an (optional) 
constraint function which will stay as a constraint in the penalized 
subproblems. We stress that this framework is very general and gives
us some flexibility to deal with different situations. The most natural
choices are probably the following ones:
\begin{enumerate}
   \item Penalize all contraints. This full penalization approach is
      probably the simplest and most straightforward approach where, 
      formally, we set $ p_{\nu} = 0 $ for every player. The resulting
      subproblems are unconstrained NEPs and are therefore, in principle,
      simple to solve. Note that, since we use an augmented Lagrangian
      method, these subproblems are still smooth in contrast to the
      (exact) penalty schemes investigated in 
      \cite{Facchinei2010a,Fukushima2011}.
   \item Another natural splitting is the case where $ h^{\nu} $ covers
      all constraints that depend on $ x^{\nu} $ only, whereas $ g^{\nu} $
      subsumes the remaining constraints. The resulting penalized problems
      then become standard (constrained) NEPs and are therefore easier
      to solve than the given GNEP since the (presumably) difficult 
      constraints are moved to the objective function.
   \item Finally, the functions $ h^{\nu} $ might, in addition to those
      constraints depending on $ x^{\nu} $ only, also contain some
      constraints that depend on the whole vector $ x $, like some 
      joint constraints for all players. The advantages is that these
      constraints might yield a compact feasible set, so this approach
      might be useful to guarantee the solvability of the resulting 
      subproblems. The latter are, in general, more complicated in this case,
      but might still be easier than the original GNEP, for example,
      in the particular case where the penalized subproblem becomes
      a jointly-convex GNEP.
\end{enumerate}
In any case, from now on, we consider GNEPs where player $ \nu $ has
to solve problems of the form \eqref{Eq:OriginalGNEP} (recall that $ h^{\nu} $ 
might not exist). Since we perform a partial penalization 
of \eqref{Eq:OriginalGNEP}, we obtain a penalized GNEP where each player 
$\nu$ has to solve the optimization problem
\begin{equation}\label{Eq:PenalizedGNEP}
   \min_{x^{\nu}} \ L_a^{\nu} (x, u^{\nu}; \rho_{\nu}) \quad
   \st \quad h^{\nu} (x) \leq 0
\end{equation}
for some parameters $u^{\nu}$ and $\rho_{\nu}$ which will typically vary 
in each iteration. The function $L_a^{\nu}$ is the augmented Lagrangian of
player $ \nu $. A typical choice is
\begin{equation*}
   L_a^{\nu}(x,u;\rho)=\theta_{\nu}(x)+\frac{\rho}{2}\left\|\left(g^{\nu}(x)+
   \frac{u}{\rho}\right)_+\right\|^2,
\end{equation*}
which is the classical Powell-Hestenes-Rockafellar augmented Lagrangian 
(see \cite{Rockafellar1974}) of the optimization problem
\begin{equation*}
   \min_{x^{\nu}} \ \theta_{\nu}(x)\quad\st\quad g^{\nu}(x)\le 0.
\end{equation*}
Note that multiple variants of $L_a^{\nu}$ exist in the literature.

We proceed by stating our algorithmic framework. Whenever there is a 
sequence such as $(\lambda^k)$ which consists of components for each 
player, we will indicate the sequences of each player by $(\lambda^{\nu,k})$. 
That is, we have $ \lambda^k=\left(\lambda^{1,k},\ldots,\lambda^{N,k}\right) $.
We use this notation whenever applicable.

\begin{algorithm2}\label{Alg:GeneralAlgorithm} (Augmented Lagrangian method for GNEPs)
\begin{itemize}
   \item[(S.0)] Choose $(x^0,\lambda^0,\mu^0)\in\mathbb{R}^{n+m+p}$. Let $u^{\max}\ge 0$, $\tau_{\nu}\in(0,1)$, $\gamma_{\nu}>1$, $\rho_{\nu,0}>0$ for all $\nu=1,\ldots,N$, and set $k:=0$.
   \item[(S.1)] If $ (x^k, \lambda^k, \mu^k) $ is an approximate KKT point 
      of the GNEP: STOP.
   \item[(S.2)] Compute an approximate KKT point (to be defined below)
      $ ( x^{k+1}, \mu^{k+1} ) $ of the GNEP consisting of the minimization 
      problems
      \begin{equation}\label{Eq:GeneralAlgorithmS2}
         \min_{x^{\nu}} \ L_a^{\nu} (x, u^{\nu,k}; \rho_{\nu,k}) 
         \quad \st \quad h^{\nu} (x) \leq 0
      \end{equation}
      for each player $\nu=1,\ldots,N$.
   \item[(S.3)] For $\nu=1,\ldots,N$, update the vector of multipliers to
      \begin{equation}\label{Eq:GeneralAlgorithmS3}
         \lambda^{\nu,k+1}=\left(u^{\nu,k}+\rho_{\nu,k}g^{\nu}(x^{k+1})\right)_+.
      \end{equation}
   \item[(S.4)] For all $ \nu = 1, \ldots, N $, if
      \begin{equation}\label{Eq:GeneralAlgorithmS4}
         \big\| \min \{ - g^{\nu} (x^{k+1}), \lambda^{\nu, k+1} \} 
         \big\| \leq \tau_{\nu} \big\| \min \{ - g^{\nu} (x^k), 
         \lambda^{\nu,k} \} \big\|,
      \end{equation}
      then set $ \rho_{\nu, k+1} := \rho_{\nu,k} $. Else, set
      $ \rho_{\nu, k+1} := \gamma_{\nu} \rho_{\nu,k} $.
   \item[(S.5)] Set $ u^{k+1}=\min\{\lambda^{k+1},u^{\max}\}$, 
      $ k \leftarrow k+1 $, and go to (S.1).
\end{itemize}
\end{algorithm2}

\noindent
Some comments are due. First among them is the fact that the objective 
functions in \eqref{Eq:GeneralAlgorithmS2} are continuously differentiable,
and their gradients are given by
\begin{equation*}
   \nabla L_a^{\nu}(x,u;\rho)=\nabla\theta_{\nu}(x)+
   \nabla g^{\nu}(x)\left(u+\rho g^{\nu}(x)\right)_+;
\end{equation*}
a similar expression holds for the partial gradients with respect to 
$ x^{\nu} $. Note that $ L_a^{\nu} $ is, in general, not twice differentiable
even if all functions involved in our GNEP from \eqref{Eq:OriginalGNEP} are
twice continuously differentiable, however, the above expression of 
the gradient clearly shows that the gradient of $ L_a^{\nu} $ is still
(strongly) semismooth, see, e.g., \cite{FacchineiPang2003} for more details.

Secondly, it should be noted that the sequence $(u^k)$ plays an essential 
role in the algorithm. Due to the formula in Step 5, it is natural to think 
of $u^k$ as a safeguarded analogue of $\lambda^k$. In fact, the boundedness 
of $(u^k)$ is the single property which is most important to our convergence 
theory. Furthermore, note that the algorithm reduces to a standard quadratic
penalty method if we set $u^{\max}=0$. In practice, however, it is much more
desirable to set $u^{\max}$ to some fixed large value; we will revisit this
matter when discussing the numerical results in \Cref{Sec:Numerics}.

Our third comment is a practical one. Clearly, the main cost for a single 
iteration of \cref{Alg:GeneralAlgorithm} lies in Step 2, where we 
have to (approximately) solve a penalized GNEP. Hence, the overall 
feasibility of the method crucially depends on the solution of these 
subproblems. In an ideal scenario, we are able to compute approximate 
solutions for the penalized GNEPs relatively cheaply. However, we are yet 
to specify what we mean by "approximate solutions". To this end, consider 
the following assumption.

\begin{assumption}\label{Asm:General}
At Step 2 of \cref{Alg:GeneralAlgorithm}, we obtain $(x^{k+1},\mu^{k+1})\in\mathbb{R}^{n+p}$ with
\begin{gather*}
   \left\|\nabla_{x^{\nu}} L_a^{\nu}(x^{k+1},u^{\nu,k};\rho_{\nu,k})+
   \nabla_{x^{\nu}}h^{\nu}(x^{k+1})\mu^{\nu,k+1}\right\|\le\varepsilon_k \\
   \|\min\{-h^{\nu}(x^{k+1}),\mu^{\nu,k+1}\}\|\le\varepsilon'_k
\end{gather*}
for every $\nu$. Here, $(\varepsilon_k)\subset\mathbb{R}_+$ is bounded 
and $(\varepsilon'_k)\subset\mathbb{R}_+$ tends to zero.
\end{assumption}

\noindent
Of course, when dealing with optimality theorems, we will make the 
additional assumption that $\varepsilon_k\to 0$.

At first glance, it seems that \cref{Asm:General} is nothing but 
an approximate KKT condition for the subproblem given 
by \eqref{Eq:GeneralAlgorithmS2}. However, we do not require the multipliers 
$\mu^{\nu,k}$ to be nonnegative. This is because the second condition 
already implies that $\liminf_{k\to\infty} \mu^{\nu,k} \ge 0 $ for every $\nu$, 
where the limit is understood component-wise. In other words, every 
limit point of the sequence $(\mu^{\nu,k})$ must be nonnegative, but the 
values $\mu^{\nu,k}$ themselves are allowed to be negative. This has the 
benefit that, when computing approximate solutions of 
\eqref{Eq:GeneralAlgorithmS2}, we allow the solutions to be inexact even 
in the sense that the multipliers could become negative. From a practical
point of view, this difference plays some role because it allows, for 
example, the application of semismooth Newton-type methods for the
inexact solution of the resulting penalized subproblems (which, in general,
do not guarantee the nonnegativity of the multiplier estimates).

Let us also stress that we do not assume that we solve (or approximately
solve) the penalized subproblems in (S.4), only (approximate) KKT points
are required. This is of particular importance since, in principle, our
method should be able to deal with nonconvex problems, i.e.\ with
GNEPs which, in general, are neither player-convex nor jointly-convex.
Of course, this general setting does not allow us to get solutions of
the original GNEP, but the subsequent convergence theory still shows
that we get something useful as limit points.

As a final note, it is evident that \cref{Asm:General} can be 
simplified in the case of full penalization. Here, we can equivalently 
state the assumption as
\begin{equation*}
   \left\|\nabla_{x^{\nu}} L_a^{\nu}(x^{k+1},u^{\nu,k};\rho_{\nu,k})
   \right\|\le\varepsilon_k
\end{equation*}
and omit the auxiliary parameters $\mu^{\nu,k}$ and $\varepsilon'_k$.

\section{Convergence Analysis}\label{Sec:Convergence}

We proceed with a thorough convergence analysis for 
\cref{Alg:GeneralAlgorithm}. The analysis is split into two parts: 
one which deals with the feasibility of limit points and one which deals 
with optimality. Throughout this section, we will implicitly assume that 
the method generates an infinite sequence $(x^k)$, i.e.\ the stopping 
criterion in Step 1 of \cref{Alg:GeneralAlgorithm} is never satisfied.

\subsection{Feasibility}

A central question in all penalty- and augmented Lagrangian-type schemes
is the feasibility of limit points. This problem also arises for
standard optimization problems. Due to some recent results in this
area, see \cite{Birgin2014} and references therein, it turns out that
augmented Lagrangian methods have a very favourable property regarding
feasibility, namely that, under mild conditions, every limit point
has a minimizing property with respect to the constraint violation.

Here we try to find a counterpart of this result for GNEPs that will also
play a central role within our subsequent optimality results. It turns
out that this counterpart is a secondary GNEP defined by the constraint
functions $ g^{\nu} $ and $ h^{\nu} $ alone, where player $ \nu $ has
to solve the optimization problem
\begin{equation}\label{Eq:FeasibilityGNEP}
   \min_{x^{\nu}} \ \|g_+^{\nu}(x)\|^2 \quad \st \quad h^{\nu}(x)\le 0.
\end{equation}
We will refer to this problem as the {\em Feasibility GNEP} since it
describes the best we can expect regarding the feasibility of the
limit points: player $ \nu $ minimizes the violation of the penalized
constraints given by $ g^{\nu} $ (with respect to his own variables 
$ x^{\nu} $) under the non-penalized constraints described by $ h^{\nu} $.

We will now see that the behaviour of \cref{Alg:GeneralAlgorithm} 
crucially depends on the structure of this auxiliary problem. More precisely, 
under certain assumptions, every limit point of our algorithm is a solution
of the Feasibility GNEP.

\begin{lemma}\label{Lem:Feasibility}
Let $(x^k)$ be generated by \cref{Alg:GeneralAlgorithm} under 
\cref{Asm:General} and let $\bar{x}$ be a limit point 
of $(x^{k+1})_K$ for some $K\subset\mathbb{N}$. Then there are multipliers 
$(\hat{\mu}^{k+1})$, $k\in K$, such that the approximate KKT conditions
\begin{gather}\label{Eq:ConcreteFeasibilityAKKT}
   \nabla_{x^{\nu}}\|g_+^{\nu}(x^{k+1})\|^2+\nabla_{x^{\nu}}h^{\nu}(x^{k+1})
   \hat{\mu}^{\nu,k+1}\to_K 0 \\
   \nonumber \min\{-h^{\nu}(x^{k+1}),\hat{\mu}^{\nu,k+1}\}\to_K 0
\end{gather}
of \eqref{Eq:FeasibilityGNEP} hold for every $\nu$.
\end{lemma}

\begin{proof}
Let $\nu\in\{1,\ldots,N\}$. Clearly, \cref{Asm:General} implies 
that $h^{\nu}(\overline{x})\le 0$. If the sequence $(\rho_{\nu,k})$ is 
bounded, \eqref{Eq:GeneralAlgorithmS4} implies $g^{\nu}(\overline{x})\le 0$. 
Hence, in this case, \eqref{Eq:ConcreteFeasibilityAKKT} follows by simply 
setting $\hat{\mu}^{\nu,k+1}:=0$. Assume now that $(\rho_{\nu,k})$ is 
unbounded. For $k\in K$, consider the sequence $(\mu^{\nu,k+1})$ from 
\cref{Asm:General} and define
\begin{equation*}
	\alpha^k=\nabla_{x^{\nu}}\theta_{\nu}(x^{k+1})+\nabla_{x^{\nu}}g^{\nu}(x^{k+1}) 
	(u^{\nu,k}+\rho_{\nu,k}g^{\nu}(x^{k+1}))_+ +
	\nabla_{x^{\nu}}h^{\nu}(x^{k+1})\mu^{\nu,k+1}.
\end{equation*}
By \cref{Asm:General}, $(\alpha^k)$ is bounded. Dividing by 
$\rho_{\nu,k}$, we see that
\begin{equation*}
	\frac{\alpha^k}{\rho_{\nu,k}}=\frac{1}{\rho_{\nu,k}}\nabla_{x^{\nu}}
	\theta_{\nu}(x^{k+1})+\nabla_{x^{\nu}}g^{\nu}(x^{k+1}) 
	\left(\frac{u^{\nu,k}}{\rho_{\nu,k}}+g^{\nu}(x^{k+1})\right)_+ +
	\nabla_{x^{\nu}}h^{\nu}(x^{k+1})\frac{\mu^{\nu,k+1}}{\rho_{\nu,k}}
\end{equation*}
approaches zero. For every $i$ with $g_i^{\nu}(\overline{x})<0$, we have 
$(u_i^{\nu,k}/\rho_{\nu,k}+g_i^{\nu}(x^{k+1}))_+=0$ for sufficiently large 
$k \in K$. Hence, we obtain
\begin{equation*}
	\sum_{g_i^{\nu}(\overline{x})\ge 0}\max\left\{0,\frac{u_i^{\nu,k}}{\rho_{\nu,k}}+
	g_i^{\nu}(x^{k+1})\right\}\nabla_{x^{\nu}}g_i^{\nu}(x^{k+1})+\nabla_{x^{\nu}}
	h^{\nu}(x^{k+1})\frac{\mu^{\nu,k+1}}{\rho_{\nu,k}}\rightarrow 0.
\end{equation*}
Since $u_i^{\nu,k}/\rho_{\nu,k}\rightarrow 0$ by the boundedness of
$ ( u^{\nu, k} ) $, this implies that
\begin{equation}\label{Eq:FeasGNEP1}
	\sum_{g_i^{\nu}(\overline{x})\ge 0}g_i^{\nu}(x^{k+1})\nabla_{x^{\nu}}g_i^{\nu}(x^{k+1})
	+\nabla_{x^{\nu}}h^{\nu}(x^{k+1})\frac{\mu^{\nu,k+1}}{\rho_{\nu,k}}\rightarrow 0.
\end{equation}
Let us define
\begin{equation*}
	I^{\nu} (\overline{x}) := \big\{ i \mid g_i^{\nu} (\overline{x}) \geq 0 
	\big\}, \quad
	I^{\nu} (x^{k+1}) := \big\{ i \mid g_i^{\nu} (x^{k+1}) \geq 0 \big\}.
\end{equation*}
Then $ I^{\nu} (x^{k+1}) \subseteq I^{\nu} (\overline{x}) $ for all $ k \in K $
sufficiently large. Furthermore, let 
\begin{equation*}
	\hat \mu^{\nu, k+1} := \frac{\mu^{\nu, k+1}}{\rho_{\nu,k}} \quad
	\forall k \in K.
\end{equation*}
Then \cref{Asm:General} immediately shows that the second
part of \eqref{Eq:ConcreteFeasibilityAKKT} holds. Furthermore, the
first part also holds since
\begin{eqnarray*}
	\lefteqn{\nabla_{x^{\nu}} \frac{1}{2} \big\| g_+^{\nu} (x^{k+1}) \big\|^2 +
		\nabla_{x^{\nu}} h^{\nu} (x^{k+1}) \hat \mu^{\nu,k+1}} \\
	& = & \sum_{i \in I^{\nu} (x^{k+1})} g_i^{\nu} (x^{k+1}) \nabla_{x^{\nu}}
	g_i^{\nu} (x^{k+1}) + \nabla_{x^{\nu}} h^{\nu} (x^{k+1}) \hat \mu^{\nu,k+1} \\
	& = & \underbrace{\sum_{i \in I^{\nu} (\overline{x})} g_i^{\nu} (x^{k+1}) 
		\nabla_{x^{\nu}} g_i^{\nu} (x^{k+1}) + \nabla_{x^{\nu}} h^{\nu} (x^{k+1}) 
		\hat \mu^{\nu,k+1}}_{\to_K 0 \text{ by \eqref{Eq:FeasGNEP1}}} \\
	& & \quad - \sum_{i \in I^{\nu} ( \overline{x} ) \setminus I^{\nu} ( x^{k+1})} 
	\underbrace{g_i^{\nu} (x^{k+1}) \nabla_{x^{\nu}} 
		g_i^{\nu} (x^{k+1})}_{\to_K 0 \text{ since } g_i^{\nu}(\overline{x}) = 0} \\
	& \to_K & 0.
\end{eqnarray*}
This completes the proof.
\end{proof}

\noindent
Clearly, \eqref{Eq:ConcreteFeasibilityAKKT} is an approximate KKT 
condition which we have already encountered in 
\cref{Thm:SequentialCPLDnu,Thm:SequentialMFCQnu}. This 
immediately yields the following corollary.

\begin{corollary}\label{Thm:FeasibilityCQ}
Let $(x^k)$ be generated by \cref{Alg:GeneralAlgorithm} under 
\cref{Asm:General}, $\bar{x}$ be a limit point of $(x^k)$ and assume
that, for every $\nu$, the function $h^{\nu}$ satisfies CPLD$_{\nu}$ 
in $\bar{x}$. Then $\bar{x}$ is a KKT point of the 
Feasibility GNEP \eqref{Eq:FeasibilityGNEP}.
\end{corollary}

\begin{proof}
This is a direct consequence of \cref{Lem:Feasibility} and 
\cref{Thm:SequentialCPLDnu}.
\end{proof}

\noindent
The above results establish the aforementioned connection between 
\cref{Alg:GeneralAlgorithm} and the Feasibility GNEP. Hence, 
it is natural to ask for the solution set of this auxiliary problem. 
Clearly, every feasible point of the original GNEP is a solution 
of \eqref{Eq:FeasibilityGNEP}, since the objective functions are zero. 
The converse is not true, in general, unless we assume some regularity 
conditions. The most important example is given in the following theorem.

\begin{theorem}\label{Thm:FeasibilityMFCQ}
Let $\bar{x}$ be a KKT point of the Feasibility GNEP and assume that 
the original GNEP satisfies GNEP-EMFCQ in $\bar{x}$. Then we have 
$g^{\nu}(\bar{x})\le 0$ for every $\nu$, i.e., $ \bar{x} $ is
feasible for the GNEP from \eqref{Eq:OriginalGNEP}; in particular, 
$ \bar{x} $ is a solution of the Feasibility 
GNEP \eqref{Eq:FeasibilityGNEP}.
\end{theorem}

\begin{proof}
Assume that there is a $\nu\in\{1,\ldots,N\}$ and an 
$\ell\in\{1,\ldots,m_{\nu}\}$ such that $g_{\ell}^{\nu}(\bar{x})>0$. 
By assumption, there are multipliers $w^{\nu}\in\mathbb{R}^{p_{\nu}}$ such that
\begin{equation*}
   \nabla_{x^{\nu}}\|g_+^{\nu}(\bar{x})\|^2+\nabla_{x^{\nu}}h^{\nu}
   (\bar{x})w^{\nu}=0 \quad \text{and} \quad
   \min\{-h^{\nu}(\bar{x}),w^{\nu}\}=0
\end{equation*}
holds. After removing some vanishing terms, we obtain
\begin{equation*}
   2\sum_{g_i^{\nu}(\bar{x})>0}g_i^{\nu}(\bar{x})\nabla_{x^{\nu}}g_i^{\nu}
   (\bar{x})+\sum_{h_j^{\nu}(\bar{x})=0}w_j^{\nu}\nabla_{x^{\nu}}h_j^{\nu}
   (\bar{x})=0.
\end{equation*}
Premultiplication of this equation with $d^{\nu}$, where $d^{\nu}$ is the 
vector from GNEP-EMFCQ, yields a contradiction.
\end{proof}

\noindent
The above theorem shows that our convergence theory naturally comprises 
GNEP-EMFCQ. Of course, we could easily have carried out our analysis without 
even considering the (weaker) GNEP-CPLD. However, we believe that the 
theorems above together with the Feasibility GNEP most clearly explain the 
structure and behaviour of \cref{Alg:GeneralAlgorithm}, especially with
regard to our GNEP-tailored constraint qualifications.

Another interesting case in which the Feasibility GNEP has some structural 
properties is the following, which covers, as a special case, the 
jointly-convex GNEP. Assume that the functions $g^{\nu}$ describe a shared 
constraint (which we denote by $g$) and that $h^{\nu}$ is a function of 
$x^{\nu}$ only. Furthermore, assume that both $g$ and $h^{\nu}$ are convex. 
Hence, player $\nu$'s optimization problem takes the form
\begin{equation}\label{Eq:SpecialGNEPJC}
   \min_{x^{\nu}} \ \theta_{\nu} (x) \quad
   \st \quad g(x)\le 0,\quad h^{\nu}(x^{\nu})\le 0.
\end{equation}
For such GNEPs, we can prove the following theorem which makes the same 
assertion as \cref{Thm:FeasibilityMFCQ}. Note, however, that we do 
not require any further constraint qualifications, particularly for the 
function $g$.

\begin{theorem}\label{Thm:FeasibilityJC}
Consider a GNEP of the form \eqref{Eq:SpecialGNEPJC} with $ g, h^{\nu} $
being convex, and assume that the  GNEP has feasible points. Then, if
$\bar{x}$ is a KKT point of the corresponding Feasibility GNEP,
we have $g(\bar{x})\le 0$, i.e.\ $ \bar{x} $ is feasible for
\eqref{Eq:SpecialGNEPJC}.
\end{theorem}

\begin{proof}
Since $\bar{x}$ is a KKT point of the Feasibility GNEP, there are 
multipliers $w^{\nu}$ such that
\begin{equation*}
   \nabla_{x^{\nu}}\|g_+(\bar{x})\|^2+\nabla h^{\nu}(\bar{x}^{\nu})
   w^{\nu}=0 \quad \text{and} \quad
   \min\{-h^{\nu}(\bar{x}^{\nu}),w^{\nu}\}=0
\end{equation*}
for every $\nu$. Hence, $\bar{x}$ together with $w=(w^1,\ldots,w^N)$ 
is a KKT point of the optimization problem
\begin{equation*}
   \min \ \|g_+(x)\|^2 \quad \st \quad h^1(x^1)\le 0,~\ldots,~h^N(x^N)\le 0.
\end{equation*}
Note that this is a convex optimization problem. Hence the KKT point
is a global minimum of this minimization problem. By assumption, however, 
the feasible set of \eqref{Eq:SpecialGNEPJC} is nonempty. This implies
that $ g_+(\bar{x}) = 0 $, hence the assertion follows.
\end{proof}

\noindent
The results in this section have shown that \cref{Alg:GeneralAlgorithm}
does (in some sense) tend to achieve feasibility. However, it should be
noted that our analysis does not exclude the possibility of the sequence
$(x^k)$ converging to an infeasible point. For instance, the Feasibility
GNEP could have solutions which are not feasible for \eqref{Eq:OriginalGNEP}.
This is particularly plausible if GNEP-EMFCQ is not satisfied or the
constraint functions $g^{\nu}$, $h^{\nu}$ are not convex.

\subsection{Optimality}

We proceed by discussing the optimality of limit points of
\cref{Alg:GeneralAlgorithm} applied to the general GNEP from
\eqref{Eq:OriginalGNEP}. To this end, we recall \cref{Asm:General}.
If we assume $\varepsilon_k\to 0$, the assumption can be stated as
\begin{gather*}
   \nabla_{x^{\nu}} L_a^{\nu}(x^{k+1},u^{\nu,k};\rho_{\nu,k})+\nabla_{x^{\nu}}
   h^{\nu}(x^{k+1})\mu^{\nu,{k+1}}\to 0,\\
   \min\{-h^{\nu}(x^{k+1}),\mu^{\nu,{k+1}}\}\to 0.
\end{gather*}
By expanding the augmented Lagrangian, we obtain
\begin{equation}\label{Eq:OriginalGNEPAKKT}
   \nabla_{x^{\nu}}\theta_{\nu}(x^{k+1})+\nabla_{x^{\nu}}g^{\nu}(x^{k+1})
   \lambda^{\nu,{k+1}}+\nabla_{x^{\nu}} h^{\nu}(x^{k+1})\mu^{\nu,{k+1}}\to 0,
\end{equation}
which already suggests that the sequence $x^k$ satisfies an approximate 
KKT condition for the GNEP \eqref{Eq:OriginalGNEP}. In fact, we can prove 
the following lemma.

\begin{lemma}\label{Lem:Optimality}
Let $(x^k)$ be a sequence generated by \cref{Alg:GeneralAlgorithm} 
under \cref{Asm:General}, where $\varepsilon_k\downarrow 0$, and 
let $\bar{x}$ be a limit point of $(x^k)$ on some subsequence 
$K\subset\mathbb{N}$. If $\bar{x}$ is feasible, we have
\begin{gather*}
   \nabla_{x^{\nu}}\theta_{\nu}(x^k)+\nabla_{x^{\nu}}g^{\nu}(x^k)\lambda^{\nu,k}+
   \nabla_{x^{\nu}} h^{\nu}(x^k)\mu^{\nu,k}\to_K 0 \\
   \min\{-g^{\nu}(x^k),\lambda^{\nu,k}\}\to_K 0,\quad 
   \min\{-h^{\nu}(x^k),\mu^{\nu,k}\}\to_K 0.
\end{gather*}
for every $\nu$.
\end{lemma}

\begin{proof}
We only need to prove the second assertion. To this end, let $\nu$ and $i$ 
be given indices such that $g_i^{\nu}(\bar{x})<0$. If $(\rho_{\nu,k})$ is 
bounded, \eqref{Eq:GeneralAlgorithmS4} implies that $\lambda_i^{\nu,k}\to_K 0$. 
On the other hand, if $(\rho_{\nu,k})$ is unbounded, the updating scheme in 
\eqref{Eq:GeneralAlgorithmS3} also implies $\lambda^{\nu,k}\to_K 0$.
\end{proof}

\noindent
The above theorem shows that, barring the feasibility of $\bar{x}$, 
the sequence $(x^k)_K$ satisfies the approximate KKT conditions from 
\cref{Thm:SequentialCPLDnu}. Hence, we can use this fact to prove 
the optimality theorem below. Note that we need to explicitly assume the 
feasibility of $\bar{x}$. In some cases, this is not necessary -- 
consider, for instance, the setting of \cref{Thm:FeasibilityMFCQ}, 
where we have GNEP-EMFCQ.

\begin{theorem}\label{Thm:OptimalityCQ}
Let $(x^k)$ be a sequence generated by \cref{Alg:GeneralAlgorithm} 
under \cref{Asm:General}, where $\varepsilon_k\downarrow 0$, and 
let $\bar{x}$ be a limit point of $(x^k)$. Assume that one of the 
following conditions is satisfied:
\begin{enumerate}[label=\textnormal{(\alph*)}]
   \item $\bar{x}$ is feasible and GNEP-CPLD holds in $\bar{x}$.
   \item GNEP-EMFCQ holds in $\bar{x}$.
\end{enumerate}
Then $\bar{x}$ is a KKT point of the GNEP.
\end{theorem}

\begin{proof}
First assume that (a) holds. Since $ \bar{x} $ is feasible,
we can apply \cref{Lem:Optimality} and obtain a sequence
of approximate KKT points for the GNEP from \eqref{Eq:OriginalGNEP}.
The statement then follows from \cref{Thm:SequentialCPLDnu} by
using the fact that we have $ c^{\nu} = \big( g^{\nu}, h^{\nu} \big) $.

Next consider case (b). Since GNEP-EMFCQ implies GNEP-CPLD, it follows
that, for each player $ \nu $, CPLD$_{\nu}$ holds for 
$ c^{\nu} = (g^{\nu}, h^{\nu}) $. This, by definition, yields that, for
each $ \nu = 1, \ldots, N $, CPLD$_{\nu} $ holds for $ h^{\nu} $. Hence
\cref{Thm:FeasibilityCQ} shows that $ \bar{x} $ is a KKT point 
of the Feasibility GNEP \eqref{Eq:FeasibilityGNEP}. Consequently, we obtain 
from \cref{Thm:FeasibilityMFCQ} that $ \bar{x} $ is feasible
for the GNEP \eqref{Eq:OriginalGNEP}. Then we can proceed as in part (a).
\end{proof}

\noindent
Note that, despite \cref{Lem:Optimality}, the multipliers $\lambda$ 
and $\mu$ which make $\bar{x}$ a KKT point are not necessarily limit 
points of the sequences $(\lambda^k)$ and $(\mu^k)$. This is a consequence 
of GNEP-CPLD, see \cref{Thm:SequentialCPLDnu}. However, we do get this
property if we assume GNEP-EMFCQ instead, see \cref{Thm:SequentialMFCQnu}.

Finally, without proof, we would like to briefly mention another kind of 
convergence theorem one can easily show for \cref{Alg:GeneralAlgorithm}.
In the above results, we have usually required that the sequence $(x^k)$
has a limit point. If we make the (much stronger) assumption that the 
sequence of triples $(x^k,\lambda^k,\mu^k)$ has a limit point, we obtain 
the following theorem which does not require any constraint qualifications.

\begin{theorem}\label{Thm:OptimalityTripel}
Let $(x^k)$, $(\lambda^k)$ and $(\mu^k)$ be the sequences generated by 
\cref{Alg:GeneralAlgorithm} under \cref{Asm:General}, where
$\varepsilon_k\downarrow 0$. Then every limit point of the sequence 
of triples $(x^k,\lambda^k,\mu^k)$ is a KKT point of the GNEP.
\end{theorem}

\section{Computing Variational Equilibria}\label{Sec:Variational}

We have already seen that \cref{Alg:GeneralAlgorithm} possesses 
some particular convergence properties for jointly-convex GNEPs -- consider, 
for instance, \cref{Thm:FeasibilityJC}. In this section, we present 
a modified method which is tailored towards the computation of variational 
(or normalized) equilibria, cf.\ \cite{Facchinei2010,Fischer2014,Rosen1965}.
To this end, we perform an obvious change in notation and consider a GNEP
of the form
\begin{equation}\label{Eq:OriginalGNEPJC}
   \min_{x^{\nu}} \ \theta_{\nu} (x) \quad
   \st \quad g(x)\le 0,\quad h(x)\le 0
\end{equation}
with smooth functions $ g: \mathbb R^n \to \mathbb R^m $ and
$ h: \mathbb R^n \to \mathbb R^p $ whose components are assumed to be
convex. Hence, all players share the same constraints. The most 
straightforward modification of \cref{Alg:GeneralAlgorithm} is 
to simply choose the same iteration parameters
\begin{equation*}
   \tau_{\nu}, \quad \gamma_{\nu}, \quad \rho_{\nu,0}, \quad \lambda^{\nu,0}, 
   \quad u^{\nu,0}, \quad \text{and} \quad \mu^{\nu,0}
\end{equation*}
for every player $\nu$. Looking at the updating scheme in 
\cref{Alg:GeneralAlgorithm}, this implies that the corresponding 
parameters $u^{\nu,k}$, $\lambda^{\nu,k}$, and $\rho_{\nu,k}$ will remain 
independent of $\nu$ throughout -- something which is clearly desirable 
when computing variational equilibria. For the sake of simplicity, we can 
now drop the index $\nu$ altogether and simply refer to the parameters as 
$u^k$, $\lambda^k$, $\rho_k$, and so on. This prompts us to restate the 
algorithm as follows.

\begin{algorithm2}\label{Alg:VariationalEquilibriaAlgorithm} (Augmented
Lagrangian method for variational equilibria)
\begin{itemize}
   \item[(S.0)] Let $ u^{\max} \ge 0$, $\tau \in (0,1)$, $\gamma > 1$ and 
      $\rho_0>0$. Choose $ x^0 \in \mathbb R^n$, $\lambda^0\in\mathbb{R}^m$, 
      $\mu^0\in\mathbb{R}^p, u^0 \in [0, u^{\max}]^m$, and set $ k:= 0 $.
   \item[(S.1)] If $ (x^k, \lambda^k, \mu^k) $ is an approximate KKT point 
      of the GNEP: STOP.
   \item[(S.2)] Compute an approximate KKT point $ ( x^{k+1}, \mu^{k+1} ) $ 
      of the GNEP consisting of the minimization problems
      \begin{equation}\label{Eq:VariationalEquilibriaAlgorithmS2}
         \min_{x^{\nu}} \ L_a^{\nu} (x, u^k; \rho_k) \quad \st \quad h (x) \leq 0
      \end{equation}
      for each player $\nu=1,\ldots,N$.
   \item[(S.3)] Update the vector of multipliers to
      \begin{equation}\label{Eq:VariationalEquilibriaAlgorithmS3}
         \lambda^{k+1}=\left(u^k+\rho_k g^{\nu}(x^{k+1})\right)_+.
      \end{equation}
   \item[(S.4)] If
      \begin{equation}\label{Eq:VariationalEquilibriaAlgorithmS4}
         \big\| \min \{ - g (x^{k+1}), \lambda^{k+1} \} 
         \big\| \leq \tau \big\| \min \{ - g (x^k), 
         \lambda^k \} \big\|,
      \end{equation}
      then set $ \rho_{k+1} := \rho_k $. Else, set
      $ \rho_{k+1} := \gamma \rho_k $.
   \item[(S.5)] Set $ u^{k+1}=\min\{\lambda^{k+1},u^{\max}\}$, $ k 
      \leftarrow k+1 $, and go to (S.1).
\end{itemize}
\end{algorithm2}

\noindent
Clearly, \cref{Alg:VariationalEquilibriaAlgorithm} is nothing but a
special instance of \cref{Alg:GeneralAlgorithm}. Hence, the convergence
theory established in Section 4 remains valid. However, we can use the
fact that we have unified sequences (for both multipliers and penalty
parameters) to prove different convergence theorems. Before we do so,
we should revisit the subproblems which occur in Step 2. With the
understanding that we are looking for variational equilibria, it is 
natural to make the following assumption.

\begin{assumption}\label{Asm:VariationalEquilibria}
At Step 2 of \cref{Alg:VariationalEquilibriaAlgorithm}, we 
obtain $x^{k+1}\in\mathbb{R}^n$ and $\mu^{k+1}\in\mathbb{R}^p$ with
\begin{gather*}
   \left\|\nabla_{x^{\nu}} L_a^{\nu}(x^{k+1},u^k;\rho_k)+\nabla_{x^{\nu}}h(x^{k+1})
   \mu^{k+1}\right\|\le\varepsilon_k \\
   \quad \|\min\{-h(x^{k+1}),\mu^{k+1}\}\|\le\varepsilon'_k
\end{gather*}
for every $\nu$. Here, $(\varepsilon_k)\subset\mathbb{R}_+$ is bounded and 
$(\varepsilon'_k)\subset\mathbb{R}_+$ tends to zero.
\end{assumption}

\noindent
Note that \cref{Asm:VariationalEquilibria} is, essentially, a 
refined version of \cref{Asm:General}. The key difference is 
that $\mu^k$ is independent of the player index $\nu$.

We now turn to a brief convergence analysis for 
\cref{Alg:VariationalEquilibriaAlgorithm}. To this end, recall that
we have used the GNEP-CPLD constraint qualification for an analysis of \cref{Alg:GeneralAlgorithm}. Furthermore, the discussion in \Cref{Sec:CQs}
shows that, in general, this is a condition which is independent of CPLD.
Despite this fact, it turns out that we can use the classical CPLD as a
constraint qualification for \cref{Alg:VariationalEquilibriaAlgorithm}.

\begin{theorem}\label{Thm:FeasibilityVE}
Let $(x^k)$ be generated by \cref{Alg:VariationalEquilibriaAlgorithm} 
under \cref{Asm:VariationalEquilibria}, let $\bar{x}$ be a limit point
of $(x^k)$ and assume that $h$ satisfies CPLD in $\bar{x}$. Then
$\bar{x}$ is a global solution of
\begin{equation}\label{Eq:JointlyConvexFeasibilityOptP}
   \min\ \|g_+(x)\|^2 \quad\st\quad h(x)\le 0.
\end{equation}
In particular, if there are feasible points, then $\bar{x}$ is feasible.
\end{theorem}

\begin{proof}
Since $g$ and $h$ are assumed to be convex, it suffices to show that 
$\bar{x}$ is a KKT point of \eqref{Eq:JointlyConvexFeasibilityOptP}. 
To verify this, we can proceed as in the proof of \cref{Lem:Feasibility} 
and obtain a sequence $(\hat{\mu}^k)$ of multipliers such that
\begin{equation*}
   \nabla_{x^{\nu}}\|g_+(x^k)\|^2+\nabla_{x^{\nu}}h(x^k)
   \hat{\mu}^k\rightarrow_K 0 \quad \text{and} \quad
   \min\{-h(x^k),\hat{\mu}^k\}\rightarrow_K 0
\end{equation*}
for every $\nu$, where $K\subset\mathbb{N}$ is some appropriate subsequence
(note that the proof of \cref{Lem:Feasibility} shows that we can choose 
the same multipliers for each player). This implies
\begin{equation*}
   \nabla\|g_+(x^k)\|^2+\nabla h(x^k)\hat{\mu}^k\rightarrow_K 0 
   \quad \text{and} \quad
   \min\{-h(x^k),\hat{\mu}^k\}\rightarrow_K 0.
\end{equation*}
Since CPLD holds, we may assume without loss of generality that the
sequence $ \{ \hat{\mu}^k \} $ is bounded. Subsequencing if necessary,
we can therefore assume that $ \hat{\mu}^k \to_K \bar{\mu} $ for
some vector $ \bar{\mu} \in \mathbb R^p $. It then follows that
$ (\bar{x}, \bar{\mu}) $ is a KKT point of 
\eqref{Eq:JointlyConvexFeasibilityOptP}. Since this is a convex program
and there exist feasible points by assumption, the statement follows.
\end{proof}

\noindent
The proof of \cref{Thm:FeasibilityVE} clearly shows that we need the
multipliers $\hat{\mu}^k$ to be independent of $\nu$, a property which, 
in general, does not hold for the iterates generated by 
\cref{Alg:GeneralAlgorithm}. On the other hand, we do not require 
a special structure for the function $h$. In this sense, the theorem is 
actually much stronger than \cref{Thm:FeasibilityJC}.

We proceed by stating an optimality result akin to \cref{Thm:OptimalityCQ}.
Note that we do not need to explicitly assume the feasibility of the limit
point because of \cref{Thm:FeasibilityVE}.

\begin{theorem}\label{Thm:OptimalityVE}
Let $(x^k)$ be generated by \cref{Alg:VariationalEquilibriaAlgorithm} 
and $\bar{x}$ be a limit point of $(x^k)$. If $(x^k)$ satisfies 
\cref{Asm:VariationalEquilibria}, the constraints $g$ and $h$ 
permit feasible points and the function
\begin{equation*}
   x\mapsto
   \begin{pmatrix}
   g(x) \\
   h(x)
   \end{pmatrix}
\end{equation*}
satisfies CPLD in $\bar{x}$, then $\bar{x}$ is feasible and solves 
the GNEP.
\end{theorem}

\begin{proof}
Under the given assumptions, it is clear that $h$ itself also satisfies CPLD. 
Hence, by \cref{Thm:FeasibilityVE}, $\bar{x}$ is feasible. 
Furthermore, \cref{Lem:Optimality} gives us the asymptotic conditions
\begin{gather*}
   \nabla_{x^{\nu}}\theta_{\nu}(x^k)+\nabla_{x^{\nu}}g(x^k)\lambda^k+
   \nabla_{x^{\nu}} h(x^k)\mu^k\to_K 0 \\
   \min\{-g(x^k),\lambda^k\}\to_K 0,\quad \min\{-h(x^k),\mu^k\}\to_K 0.
\end{gather*}
for every $\nu$. The result then follows by concatenating these systems 
for every $\nu$ and using CPLD.
\end{proof}

\noindent
The above results are particularly interesting because the classical CPLD 
is a more amenable condition than GNEP-CPLD. For example, we have the 
well-known chain of implications
$   \text{Slater}\implies\text{MFCQ}\implies\text{CPLD} $, 
which allows us to use the (easily verifiable) Slater condition as a CQ 
for jointly-convex GNEPs.

Before we conclude this section, we would like to point out another 
property of \cref{Alg:VariationalEquilibriaAlgorithm}. The augmented
Lagrangian for player $\nu$ is given by
\begin{equation*}
   L_a^{\nu}(x,u;\rho)=\theta_{\nu}(x)+\frac{\rho}{2}\left\|\left(g(x)+
   \frac{u}{\rho}\right)_+\right\|^2.
\end{equation*}
Clearly, the second term is independent of $\nu$. This allows us to 
decompose the augmented Lagrangian in the following way:
\begin{equation*}
   L_a^{\nu}(x,u;\rho)=\theta_{\nu}(x)+P(x,u;\rho),
\end{equation*}
where $P$ is a convex penalty term which is independent of $\nu$. This 
decomposition is useful when designing methods for the solution of the 
subproblems. For instance, it is well-known that a critical property of 
(jointly-convex) GNEPs is the monotonicity of the function
\begin{equation*}
   F(x)=
   \begin{pmatrix}
   \nabla_{x^1}\theta_1(x) \\
   \vdots \\
   \nabla_{x^N}\theta_N(x)
   \end{pmatrix}.
\end{equation*}
When adding a convex penalty term to the functions $\theta_{\nu}$, 
it is easy to see that this property is preserved.

\section{Implementation and Numerical Results}\label{Sec:Numerics}

In this section, we present some empirical results to showcase the convergence of our method(s). To this end, we implement \cref{Alg:GeneralAlgorithm} in MATLAB\textsuperscript{\textregistered} and, for the sake of simplicity, we solve every problem by performing a full penalization. This is especially attractive because many of the convergence theorems (e.g.\ \ref{Thm:FeasibilityCQ} and \ref{Thm:FeasibilityJC}) hold without any further assumptions.

The test suite we use is identical to the one from \cite{Facchinei2010a}. For every problem, we use the same parameters $u^{\max}=10^6$ and $\rho_{\nu,0}=1$ for every $\nu$. The remaining parameters are chosen depending on the size of the problem:
\begin{align*}
	\tau_{\nu}=0.1, \quad \gamma_{\nu}=10, \qquad & \text{if}~n\le 100; \\
	\tau_{\nu}=0.5, \quad \gamma_{\nu}=2, \qquad  & \text{if}~n>100.
\end{align*}
This represents a quite aggressive penalization for small problems and a more cautious scheme for large problems. We have found this distinction to be very efficient for our problem set. For the computation of the initial multipliers $\lambda^{\nu,0}$ (and $u^{\nu,0}$, which we set to the same value), we recall the KKT conditions for player $\nu$, which can be stated as
\begin{equation*}
	\nabla_{x^{\nu}}\theta_{\nu}(x^0)+\nabla_{x^{\nu}}g^{\nu}(x^0)\lambda^{\nu,0}=0 \quad
	\text{and} \quad \min\{-g^{\nu}(x^0),\lambda^{\nu,0}\}=0.
\end{equation*}
We now solve the first condition in a least-squares sense by setting $\lambda_i^{\nu,0}=0$ for every $i$ with $g_i^{\nu}(x^0)<0$ and using the MATLAB\textsuperscript{\textregistered} function \texttt{lsqnonneg} to compute a nonnegative least-squares solution of
\begin{equation*}
	\nabla_{x^{\nu}}\theta_{\nu}(x^0)+\nabla_{x^{\nu}}g^{\nu}(x^0)\lambda^{\nu,0}=0.
\end{equation*}
Finally, the overall stopping criterion we use is
\begin{equation*}
	\|\nabla_{x^{\nu}}\theta_{\nu}(x)+\nabla_{x^{\nu}}g^{\nu}(x)\lambda^{\nu}\|_{\infty} \le\varepsilon,
	\quad \|g_+^{\nu}(x)\|_{\infty}\le\varepsilon,
	\quad \text{and} \quad |g^{\nu}(x)^T \lambda^{\nu}|\le\varepsilon
\end{equation*}
for every $\nu$. Here, $\varepsilon$ is some prescribed stopping tolerance which we set to $10^{-8}$.

\subsection{Solution of the subproblems}

Since we perform a full penalization, the subproblems which occur at Step 2 of \cref{Alg:GeneralAlgorithm} are unconstrained NEPs where player $\nu$'s optimization problem is given by
\begin{equation*}
	\min_{x^{\nu}} \ L_a^{\nu}(x,u^{\nu,k};\rho_{\nu,k}).
\end{equation*}
Hence, we simply solve these problems by considering the nonlinear equation
\begin{equation}\label{Eq:SubproblemF}
	F(x)=
	\begin{pmatrix}
		\nabla_{x^1} L_a^1(x,u^{1,k};\rho_{1,k}) \\
		\vdots \\
		\nabla_{x^N} L_a^N(x,u^{N,k};\rho_{N,k})
	\end{pmatrix}
	\stackrel{!}{=}0.
\end{equation}
In principle, we could use any general-purpose nonlinear equation solver to solve this equation. However, it should be noted that $F$ is, in general, a semismooth function with non-isolated solutions. Hence, special care needs to be taken when selecting an algorithm. For instance, the classical semismooth Newton method \cite{Qi1999,Qi1993} typically does not exhibit (locally) superlinear convergence for such problems, whereas more sophisticated methods such as Levenberg-Marquardt methods \cite{Fan2005,Yamashita2001} or the LP-Newton method \cite{Facchinei2014} are known to be more efficient under certain assumptions. For our numerical testing, we decided to employ a Levenberg-Marquardt type algorithm from \cite{Fan2005} where the basic step $d$ is given by
\begin{equation*}
\left(J(x)^T J(x)+\alpha(x)I\right)d=-J(x)^T F(x).
\end{equation*}
Here, $J(x)$ is some suitable (generalized) Jacobian of $F$ and $\alpha(x)=\|F(x)\|$. In order to improve the global convergence properties of this method, we have decided to combine it with a classical Levenberg-Marquardt parameter updating scheme, i.e.\ we consider the equation
\begin{equation*}
\left(J(x)^T J(x)+\alpha \|F(x)\| I\right)d=-J(x)^T F(x)
\end{equation*}
and iteratively update $\alpha$ (in a heuristic manner) based on the success of the last step. A precise statement of the algorithm is as follows.

\begin{algorithm2}(Levenberg-Marquardt type method for $F$)
	 \begin{itemize}
	 	\item[(S.0)] Let $x^0\in\mathbb{R}^n$, $\alpha_0=1$, $\varepsilon>0$, and set $k=0$.
	 	\item[(S.1)] If $\|F(x^k)\|\le\varepsilon$ holds: STOP.
	 	\item[(S.2)] Choose $V_k\in\partial F(x^k)$ and compute $d^k$ by solving
	 	\begin{equation}\label{Eq:LevenbergMarquardtEq}
	 		(V_k^T V_k+\alpha_k\|F(x^k)\|I)d^k=-V_k^T F(x^k).
	 	\end{equation}
	 	If $\|F(x^k+d^k)\|<\|F(x^k)\|$, set $\alpha_{k+1}=0.1\alpha_k$ and go to (S.4).
	 	\item[(S.3)] Iteratively set $\alpha_k\leftarrow 10\alpha_k$ and re-compute $d^k$ as given by \eqref{Eq:LevenbergMarquardtEq} until $\|F(x^k+d^k)\|<\|F(x^k)\|$. Finally, set $\alpha_{k+1}=\alpha_k$.
	 	\item[(S.4)] Set $x^{k+1}=x^k+d^k$, $k\leftarrow k+1$, and go to (S.1).
	 \end{itemize}
\end{algorithm2}

\noindent
Note that we use the same tolerance $\varepsilon=10^{-8}$ as given at the beginning of this section. Furthermore, since $F$ is only semismooth, the above is not a globally convergent algorithm. In fact, the loop in (S.3) does not necessarily terminate finitely if the current point $x^k$ is one where $F$ is not differentiable. To safeguard against this case, we terminate the loop in (S.3) if
\begin{equation*}
	\|d\|<\frac{\varepsilon}{\|V^k\|_F}.
\end{equation*}
Despite the necessity of such safeguarding techniques, we have found the above method to be sufficient for nearly all our examples.

\subsection{Numerical Results}

We now present our results. For a given problem, $N$ denotes the number of players, $n$ is the total number of variables, $k$ is the number of outer iterations, $i_{total}$ is the accumulated number of inner iterations and F denotes a failure. We also include certain values which measure the feasiblity, optimality and complementarity at the solution. These are denoted $R_f$, $R_o$ and $R_c$, respectively. The values are calculated as follows:
\begin{align*}
	R_f & = \max_{\nu=1,\ldots,N} \|g_+^{\nu}(x)\|_{\infty} \\
	R_o & = \max_{\nu=1,\ldots,N} \|\nabla_{x^{\nu}}\theta_{\nu}(x)+\nabla_{x^{\nu}}g^{\nu}(x)\lambda^{\nu}\|_{\infty} \\
	R_c & = \max_{\nu=1,\ldots,N} |g^{\nu}(x)^T \lambda^{\nu}|.
\end{align*}

\begin{table}\centering
\caption{Numerical results for \cref{Alg:GeneralAlgorithm}.}
\begin{tabular}{|lllc|cccccc|}\hline
	Example & $N$ & $n$ & $x^0$ & $k$ & $i_{total}$ & $R_f$ & $R_o$ & $R_c$ & $\rho_{\max}$ \\ \hline

	A.1   & 10 &  10 & 0.01 &   7 &    20 & 1.5e-10 & 8.9e-16 & 4.2e-11 &   100 \\
	      &    &     &  0.1 &   6 &    13 &   8e-09 & 5.9e-13 & 2.1e-09 &   100 \\
	      &    &     &    1 &   7 &    19 & 1.5e-10 & 2.9e-16 & 4.2e-11 &   100 \\ \hline
	
	A.2   & 10 &  10 & 0.01 &   9 &   108 & 4.7e-09 & 2.3e-09 & 1.4e-09 &  1000 \\
	      &    &     &  0.1 &   8 &    70 & 2.9e-09 & 4.1e-14 & 2.9e-11 &   100 \\
	      &    &     &    1 &  10 &   192 & 4.9e-10 & 5.6e-14 & 1.5e-10 & 1e+05 \\ \hline
	
	A.3   &  3 &   7 &    0 &   1 &     4 &       0 &   1e-09 &       0 &     1 \\
	      &    &     &    1 &   1 &     5 &       0 & 3.6e-15 &       0 &     1 \\
	      &    &     &   10 &   1 &     5 &       0 & 1.7e-10 &       0 &     1 \\ \hline
	
	A.4   &  3 &   7 &    0 &  12 &    63 & 2.6e-11 & 1.5e-09 & 2.6e-09 & 1e+04 \\
	      &    &     &    1 &   0 &     0 &       0 &       0 &       0 &     1 \\
	      &    &     &   10 &  11 &   202 & 2.5e-12 & 4.6e-10 & 7.4e-10 & 1e+04 \\ \hline
	
	A.5   &  3 &   7 &    0 &   8 &    20 &   2e-10 & 1.7e-13 & 4.8e-10 &  1000 \\
	      &    &     &    1 &   8 &    20 & 3.5e-10 & 4.9e-13 & 8.3e-10 &  1000 \\
	      &    &     &   10 &  10 &    27 & 6.9e-09 &   1e-13 & 6.2e-09 &  1000 \\ \hline
	
	A.6   &  3 &   7 &    0 &  14 &    68 & 1.9e-11 & 6.6e-10 & 4.2e-09 & 1e+04 \\
	      &    &     &    1 &  11 &    92 & 9.8e-12 & 4.5e-09 & 5.1e-09 & 1e+04 \\
	      &    &     &   10 &  14 &    82 & 1.9e-11 & 6.6e-10 & 4.2e-09 & 1e+04 \\ \hline
	
	A.7   &  4 &  20 &    0 &  13 &    35 & 6.6e-12 & 1.7e-11 & 2.3e-09 & 1e+04 \\
	      &    &     &    1 &  12 &    39 & 1.1e-11 & 1.4e-11 & 3.8e-09 & 1e+04 \\
	      &    &     &   10 &  12 &    52 & 4.1e-12 & 1.2e-11 & 1.7e-09 & 1e+05 \\ \hline
	
	A.8   &  3 &   3 &    0 & F & & & & &  \\
	      &    &     &    1 &   1 &     4 & 4.9e-11 & 4.9e-11 & 4.9e-11 &     1 \\
	      &    &     &   10 &   3 &    14 & 4.5e-12 & 4.9e-12 & 4.5e-12 &    10 \\ \hline
	
	A.9a  &  7 &  56 &    0 &   9 &    46 & 2.3e-09 &   8e-15 & 7.6e-09 &    10 \\
	A.9b  &  7 & 112 &    0 &  26 &    75 & 2.8e-10 &   1e-14 & 2.7e-09 &    16 \\ \hline
	
	A.10a &  8 &  24 & see \cite{Facchinei2010a} &  11 &   243 & 9.8e-13 & 4.5e-11 & 4.5e-12 & 1e+05 \\
	A.10b & 25 & 125 & see \cite{Facchinei2010a} &  19 &  2519 & 6.7e-10 & 1.8e-11 & 4.9e-09 &    64 \\
	A.10c & 37 & 222 & see \cite{Facchinei2010a} &  40 &  3658 & 7.2e-13 & 9.3e-12 & 1.6e-09 & 5e+05 \\
	A.10d & 37 & 370 & see \cite{Facchinei2010a} &  19 &  2527 & 2.9e-11 & 2.3e-12 & 3.1e-10 &   256 \\
	A.10e & 48 & 576 & see \cite{Facchinei2010a} &  18 &  4048 & 1.2e-10 & 7.1e-12 & 1.5e-09 &   256 \\ \hline
	
	A.11  &  2 &   2 &    0 &   9 &    17 & 6.4e-09 & 2.9e-15 & 3.2e-09 &    10 \\ \hline
	
	A.12  &  2 &   2 &(2,0) &   1 &     5 &       0 & 8.9e-16 &       0 &     1 \\ \hline
	
	A.13  &  3 &   3 &    0 &   4 &    20 & 3.3e-09 & 7.6e-12 & 1.9e-09 &     1 \\ \hline
	
	A.14  & 10 &  10 & 0.01 &   1 &     8 &       0 & 8.2e-14 &       0 &     1 \\ \hline
	
	A.15  &  3 &   6 &    0 &   1 &     7 &       0 & 2.8e-14 &       0 &     1 \\ \hline
	
	A.16a &  5 &   5 &   10 &  10 &    26 & 1.3e-10 &   6e-14 & 3.7e-09 &    10 \\
	A.16b &  5 &   5 &   10 &   9 &    26 & 6.1e-11 & 3.6e-15 & 1.1e-09 &    10 \\
	A.16c &  5 &   5 &   10 &   7 &    23 &   9e-10 & 1.5e-13 & 6.4e-09 &    10 \\
	A.16d &  5 &   5 &   10 &   9 &    24 &   4e-09 & 2.1e-14 & 1.9e-09 &     1 \\ \hline
	
	A.17  &  2 &   3 &    0 &   8 &    20 & 4.5e-11 & 3.4e-13 & 1.1e-10 &   100 \\ \hline
	
	A.18  &  2 &  12 &    0 &   9 &    34 & 1.3e-11 & 1.1e-11 & 2.4e-10 &  1000 \\
	      &    &     &    1 &   9 &    34 & 1.3e-11 & 1.2e-11 & 2.4e-10 &  1000 \\
	      &    &     &   10 &   9 &    32 & 1.3e-11 & 1.8e-11 & 2.4e-10 &  1000 \\ \hline
	
\end{tabular}
\end{table}

\noindent
Clearly, some remarks are in order:

\begin{enumerate}
	\item With the exception of problem A.8, the augmented Lagrangian method was able to solve every problem quite efficiently. It is particularly noteworthy that the method achieves a very high accuracy, typically in the region of $10^{-10}$. This compares quite favourably to other methods for GNEPs, such as the interior-point method from \cite{Dreves2011} or the exact penalty method from \cite{Facchinei2010a}.
	\item We have also tried the algorithm with different choices of $u^{\max}$. Recall that, for $u^{\max}=0$, the algorithm is essentially a quadratic penalty method. The following table lists some values for $u^{\max}$ and corresponding failure numbers.
	\begin{center}
	\begin{tabular}{r|ccccc}
		$u^{\max}$ & 0  & 10 & $10^2$ & $10^4$ & $10^6$ \\ \hline
		failures   & 29 & 18 & 10     & 1      & 1
	\end{tabular}
	\end{center}
	\item For most problems, the stopping accuracy tends to have little effect on the speed of the algorithm. A notable exception is problem A.2, where we observed significantly lower (by a factor of 3) iteration numbers when using a tolerance of $10^{-4}$. We suspect that this is a consequence of the very narrow feasible set in this problem (see \cite{Facchinei2010a}).
	\item Clearly, the overall speed of the algorithm crucially depends on how quickly the subproblems are solved. In this regard, the Levenberg-Marquardt algorithm seems to greatly benefit from the (semi-)smoothness of the function $F$ from \eqref{Eq:SubproblemF}. We investigated some of the problems on a sample basis and found that the Levenberg-Marquardt method appears to be superlinearly convergent for all of them. Despite this, we believe that there is a lot of room for improvement here.
	\item Another factor which greatly affects the performance of the algorithm is the choice of the parameters which handle the multipliers and penalty parameters. In this regard, our choices are quite simple and straightforward. However, for some problems, we observed that fine-tuning the parameters can yield a significant speed improvement.
	\item For problem A.8 with the starting point $x^0=0$, the subproblem algorithm is unable to compute a solution and, hence, the overall iteration breaks down. Another peculiarity of problem A.8 is that, for a suitable choice of parameters, one can get the algorithm to converge to the infeasible point $\bar{x}=(1.5,0,2)$. This point (together with its corresponding multipliers) satisfies the stationarity part of the KKT conditions, but (due to the infeasibility) is not a solution of the GNEP. Furthermore, one can easily verify that $\bar{x}$ is a solution of the Feasibility GNEP \eqref{Eq:FeasibilityGNEP}, as suggested by \cref{Thm:FeasibilityCQ}, but GNEP-EMFCQ does not hold in $\bar{x}$. This shows that the assertions of \cref{Thm:FeasibilityCQ} can, in general, not be sharpened.
\end{enumerate}

\section{Final Remarks}\label{Sec:Final}

We have introduced an augmented Lagrangian method for the solution of generalized Nash equilibrium problems. Our method is quite flexible in the sense that it allows partial penalization of constraints and can be modified for the computation of variational equilibria of jointly-convex GNEPs. The numerical testing we have done indicates that the method works quite well in practice, since it possesses good global convergence properties and easily achieves a very high accuracy, provided the problem is sufficiently well-behaved.

It should be noted that there are still many aspects which might lead to substantial numerical improvements. Aside from the fine-tuning of iteration parameters, a more detailed analysis of the subproblems which occur in our method might lead to insights on their solution. In this regard, it would be interesting to analyse whether the subproblems satisfy certain regularity conditions such as an error-bound to the solution set \cite{Dreves2014,Izmailov2014} or how other methods such as smoothing Newton methods \cite{Qi1999} could be incorporated into the solution process. Further possible extensions of the ALM are second-order multiplier iterations or approaches such as the exponential method of multipliers, cf.\ \cite{Bertsekas1982}.

On another note, the theoretical analysis of our algorithm has uncovered a series of properties and concepts which extend the rich theoretical background of augmented Lagrangian methods to the field of GNEPs. For instance, the constraint qualifications introduced in \Cref{Sec:CQs} (one of which has previously been used in the literature) are very general and hence, we hope, they will find applications in the context of other methods for multi-player games.

The same goes for our notion of the {\em Feasibility GNEP}, which is a new optimality concept for GNEPs that offers a very clear insight on the behaviour of the augmented Lagrangian method. This is a generalization of a corresponding concept for classical optimization problems, cf.\ \cite{Birgin2014}, which has enjoyed a variety of applications, e.g.\ in the context of Sequential Quadratic Programming (SQP) methods in \cite{Burke1989}. A natural continuation of this idea would be an SQP-type method for GNEPs, which we envision as a possible path for future research.

\bibliographystyle{siamplain}
\bibliography{references}

\end{document}